\documentclass[12pt]{amsart}
\usepackage{amsmath,amscd,amssymb,amsfonts,amsthm,enumitem,hyperref}
\usepackage{tikz-cd}
\usepackage{mathabx}
\hypersetup{colorlinks,   linkcolor={red!50!black},    citecolor={blue!70!black},   urlcolor={blue!80!black}}

\usepackage{graphicx}
\usepackage[cmtip,matrix,arrow]{xy}

\newtheorem{theorem}{Theorem}[section]

\newtheorem{proposition}[theorem]{Proposition}

\newtheorem{lemma}[theorem]{Lemma}
\theoremstyle{definition}    
\newtheorem{definition}[theorem]{Definition}
\theoremstyle{remark}

\newtheorem{remark}[theorem]{Remark}
\newtheorem{remarks}[theorem]{Remarks}
\newtheorem{example}[theorem]{Example}
\newcommand\A{\mathcal{A}}
\renewcommand{\AA}{\mathbb{A}}
\newcommand{\Cour}[1]      {[\![#1]\!]}

\newcommand\G{\mathcal{G}}

\newcommand{\V}{\mathcal{V}}
\renewcommand{\L}{\mathcal{L}}
\renewcommand{\O}{\mathcal{O}}
\newcommand{\T}{\mathbb{T}}

\newcommand{\ca}{\mathcal}

\newcommand{\E}{\ca{E}}

\newcommand{\R}{\mathbb{R}}
\newcommand{\C}{\mathbb{C}}

\newcommand{\Q}{{\ca Q}}
\newcommand\pt{\on{pt}}

\newcommand{\ez}{\mathsf{e}}
\newcommand{\vz}{\mathsf{v}}
\newcommand\lie[1]{\mathfrak{#1}}
\renewcommand{\k}{\lie{k}}
\newcommand{\h}{\lie{h}}
\newcommand{\g}{\lie{g}}
\newcommand{\m}{\lie{m}}

\renewcommand{\a}{\mathsf{a}}

\newcommand{\on}{\operatorname}

\newcommand{\Ad}{ \on{Ad} }
\newcommand{\ad}{\on{ad}}

\newcommand{\Hom}{ \on{Hom}}

\renewcommand{\ker}{ \on{ker}}

\newcommand{\Mult}{  \on{Mult}}

\newcommand{\da}{\dasharrow}

\newcommand\qu{/\kern-.7ex/} 
\newcommand{\fus}{\circledast} 
\newcommand{\TG}{\mathbb{T}G}
\newcommand{\lra}{\longrightarrow}
\newcommand{\hra}{\hookrightarrow}

\renewcommand{\d}{{\mbox{d}}}
\newcommand{\ol}{\overline}

\newcommand\Phinv{\Phi^{-1}}

\newcommand{\f}{\frac}

\newcommand{\p}{\partial}
\renewcommand{\l}{\langle}
\renewcommand{\r}{\rangle}
\newcommand\hh{{\f{1}{2}}}

\newcommand{\ti}{\tilde}
\newcommand{\eeq}{\end{eqnarray*}}
\newcommand{\beq}{\begin{eqnarray*}}

\newcommand{\pr}{\on{pr}}
\newcommand{\wh}{\widehat}

\newcommand{\mf}{\mathfrak}
\newcommand{\rra}{\rightrightarrows}

\renewcommand{\subset}{\subseteq}

\newcommand{\dd}{\mathfrak{d}}
\newcommand{\cc}{\mathfrak{c}}
\newcommand{\sz}{\mathsf{s}}
\newcommand{\tz}{\mathsf{t}}
\newcommand{\op}{{\on{op}}}
\begin{document}
\title{Manin pairs and moment maps revisited}
\author{Eckhard Meinrenken}
\author{Selim Tawfik}

\begin{abstract}
The notion of quasi-Poisson $G$-spaces with $D/G$-valued moment maps was introduced by Alekseev and Kosmann-Schwarzbach in 
1999. Our main result is a \emph{Lifting Theorem}, establishing a bijective correspondence between the categories of quasi-Poisson $G$-spaces with $D/G$-valued moment maps and  of quasi-Poisson $G\times G$-spaces with $D$-valued moment maps. Using this result, 
we give simple constructions of fusion and conjugation for these spaces, and new examples coming from moduli spaces. 
\end{abstract}
\maketitle

\tableofcontents

\section{Introduction}\label{sec:intro}
A Manin pair $(\dd,\g)$ is a Lie algebra $\dd$ with an $\ad$-invariant metric $\l\cdot,\cdot\r$, together with a Lagrangian
Lie subalgebra $\g$.  Suppose $(D,G)$ is a Lie group pair integrating the Manin pair $(\dd,\g)$. 
The choice of a  Lagrangian subspace $\m$ complementary to $\g$ determines the structure of a quasi Lie bialgebra, and so 
$G$ becomes a quasi-Poisson Lie group. 

In their article \emph{Manin pairs and moment maps} \cite{al:ma}, Alekseev and Kosmann-Schwarzbach introduced the 
concept of a \emph{quasi-Poisson $G$-manifold}, given by a $G$-manifold $M$ with a bivector field $\pi$, satisfying certain compatibility conditions. They also defined a notion of  \emph{moment map}
\begin{equation}\label{eq:aks} \Phi\colon M\to D/G\end{equation}
for such a quasi-Poisson $G$-action, similar to Lu's $G^*$-valued moment map for Poisson Lie group actions \cite{lu:mo}
and generalizing the group-valued moment maps of \cite{al:mom}. Subsequent contributions by Burztyn-Crainic \cite{bur:di,bur:di1} expressed $D/G$-valued moment maps within the framework of Dirac geometry; this was further simplified in \cite{bur:qua,bur:cou}.  

The present work was motivated by \v{S}evera's construction \cite{sev:lef} of a \emph{fusion product} $M_1\fus M_2$ of quasi-Poisson $G$-spaces with $D/G$-valued moment maps. \v{S}evera arrived at this fusion  operation as a consequence of 
his theory of left and  right centers for quasi-Poisson spaces. 
Similarly, he obtained the \emph{conjugate} $M^*$ of a space with $D/G$-valued moment map; such an involution 
had earlier been described by Bursztyn-Crainic \cite[Proposition 4.3]{bur:di1}. 

Aiming for a more direct  construction of fusion and conjugation, we were led to an alternative description of $D/G$-valued moment maps, by lifting the moment map to the group $D$. 
Given a $G$-equivariant map \eqref{eq:aks}, define a space $\wh{M}$ and a $G\times G$-equivariant map $\wh{\Phi}$ by the 
fiber product diagram
\begin{equation}\label{eq:diagram} \xymatrix{
\wh{M}\ar[r]^{\wh{\Phi}}\ar[d] & D\ar[d]\\
M\ar[r]_{\Phi} & D/G
}
\end{equation}
Theorem \ref{th:liftingtheorem} shows that a $G$-equivariant quasi-Poisson structure on $M$, with moment map $\Phi$,  
is equivalent to a $G\times G$-equivariant quasi-Poisson structure on $\wh{M}$, with moment map $\wh{\Phi}$.
This \emph{Lifting Theorem}   sets up a 1-1 correspondence between quasi-Poisson $G$-spaces with $D/G$-valued moment maps and 
quasi-Poisson $G\times G$-spaces with $D$-valued moment maps.

The fusion operation for $G$-spaces with $D/G$-valued moment maps may now be  described in terms of a 
fusion operation of $G\times G$-spaces with $D$-valued moment maps: 
Given two such spaces, with moment maps $\wh{\Phi}_i\colon \wh{M}_i\to D$, one defines 
\[ \wh{M}_1\fus \wh{M}_2=(\wh{M}_1\times \wh{M}_2)\big/(\{e\}\times G_{\Delta}\times \{e\}),\]
with equivariant map $\wh{\Phi}_1\fus \wh{\Phi}_2\colon \wh{M}_1\fus \wh{M}_2\to D$ induced by the pointwise product. 
Theorem \ref{th:fusion} asserts that the result is again a quasi-Poisson space. Taking quotients, one obtains a quasi-Poisson $G$-space $M_1\fus M_2$ with $D/G$-valued moment map; this space is usually different, as a $G$-space, from the direct product $M_1\times M_2$. Similarly, the conjugate $\wh{M}^*$ of a quasi-Poisson $G\times G$-space with $D$-valued moment map 
$\wh{\Phi}\colon \wh{M}\to D$ is defined  to be the same space but with the new $G\times G$-action obtained by interchanging the  two factors, and with $\wh{\Phi}^*=\wh{\Phi}^{-1}$. By 
Theorem \ref{th:conjugation} this space is again quasi-Poisson. Again, the quasi-Poisson $G$-spaces 
$M$ and $M^*=\wh{M}^*/G$  are usually different as $G$-spaces. 

We shall develop these results in the broader context of generalized moment map targets $\Q$ for Manin pairs $(\dd,\g)$, given by manifolds $\Q$ of dimension equal to that of $\g$, equipped with a  $\dd$-action with Lagrangian stabilizers. This includes the moment map target $D/G$, but also $D/U$ for any closed Lagrangian subgroup $U\neq G$, 
the dual Poisson Lie group $G^*$ (given an extension of $(\dd,\g)$ to a Manin triple), and many other examples. In particular, $D$ will be regarded as a moment map target for the Manin pair $(\ol{\dd}\oplus \dd,\g\oplus \g)$. Accordingly, the Lifting Theorem \ref{th:liftingtheorem} is proved as a general result on quotients and lifts of Hamiltonian spaces for these generalized moment map targets. 

A quasi-Poisson structure on a manifold determines a generalized foliation,  known as its quasi-symplectic foliation. The quasi-Poisson space $M$ is called  \emph{quasi-symplectic} if its components are leaves. The axioms for  quasi-symplectic spaces with $\Q$-valued moment maps may be expressed in terms of 2-forms $\omega\in \Omega^2(M)$, depending on the choice of \emph{splitting} of the Dirac structure over $\Q$. Theorem \ref{th:groupoid}  gives an explicit description of the quasi-symplectic groupoids 
 \[ G\ltimes \Q\rra \Q\]
 integrating these Dirac structures, generalizing \cite[Section 6.4]{bur:int} and \cite{bur:di1,igl:uni} 
  (for the case $\Q=D/G$).

 The moment map target $D$ has a distinguished splitting, leading to an 
 explicit description of the 2-form on the quasi-symplectic groupoid $(G\times G)\ltimes D\rra D$ integrating the Dirac structure, and to differential form descriptions of   quasi-symplectic $G\times G$-spaces with $D$-valued moment maps similar to \cite{al:mom}.   Similar descriptions for $\Q=D/G$ depend on the choice of a principal connection. 

As applications of the Lifting Theorem, one finds many examples of spaces with $D/G$-valued moment maps, 
coming from moduli spaces of flat connections.
More specifically, these examples correspond to surfaces with bi-colored boundary as illustrated below, where the black edges 
are called `free' and the red edges `colored'. 

\begin{center}
	\includegraphics[scale=0.5]{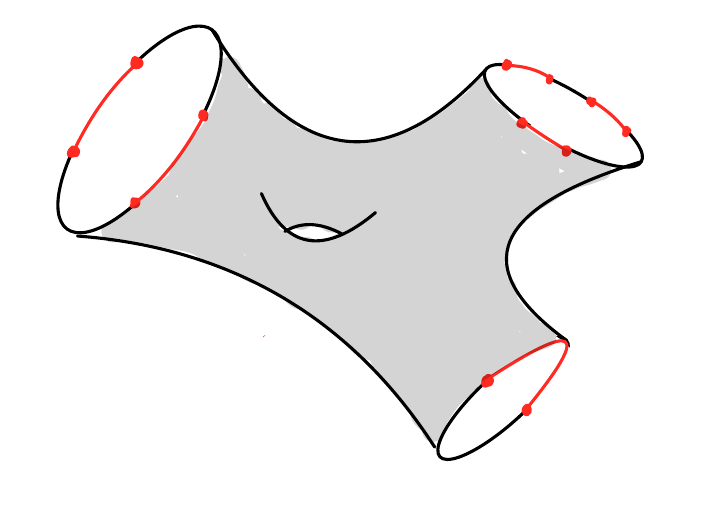}
\end{center}

The moduli space of flat $D$-bundles over such a surface, with trivializations at the vertices and with holonomy at colored edges in $G$, is a quasi-symplectic $(G\times G)^r$-space with $D^r$-valued moment map, where 
$r$ is the number of  free edges. Note that similar pictures appear in \cite{sev:lef}, with somewhat different interpretations. 

The organization of the paper is as follows. Section \ref{sec:background} reviews background material on quasi-Poisson geometry
and moment maps. We define generalized moment map targets $\Q$ for Manin pairs and develop their basic theory. 
Section \ref{sec:quasisymplectic} specializes to the quasi-symplectic case. In the presence of splittings, this leads to descriptions in terms of differential forms. In Section \ref{sec:integration}, we construct the quasi-symplectic groupoids integrating the Dirac structures on the moment map targets $\Q$. Section \ref{sec:quotientlifts} is devoted to the Lifting Theorem for moment map targets. 
As a special case, this proves the equivalence between $D$-valued moment maps and $D/G$-valued moment maps. 
This is then used in Section \ref{sec:fusion} to construct fusion and conjugation operations of spaces with $D/G$-valued moment maps in terms of their associated spaces with $D$-valued moment maps. 
Finally, Section \ref{sec:moduliexamples} gives examples from the moduli spaces of surfaces with bi-colored boundary.  The appendices discuss background material on Dirac geometry, the reduction of Dirac morphisms and a Dirac Cross Section Theorem, and the Dirac-geometric approach to the moduli space examples. 
\bigskip

\noindent{\bf Acknowledgments.} The first author is indebted to Pavol \v{S}evera for sharing his thoughts on fusion of $D/G$-valued moment maps, in an email exchange from October 2013. These ideas appeared in his article \cite{sev:lef}, and motivated the 2021 Ph.D. thesis \cite{taw:th} of the second author. We also thank Anton Alekseev, Daniel \'{A}lvarez, Henrique Bursztyn, and Li Yu for helpful comments. Research of E. Meinrenken was supported by an NSERC Discovery Grant.

\section{Moment maps for Manin pairs}\label{sec:background}
We begin by recalling the definition of Alekseev and Kosmann-Schwarzbach of quasi-Poisson spaces, using the reformulation in Dirac geometric terms. See the appendix for background information on 
Courant algebroids, Dirac structures, and their morphisms. Note that what we call Dirac morphisms are also referred to as \emph{strong Dirac morphisms} or \emph{morphisms of Manin pairs}.

\subsection{Quasi-Poisson manifolds}
A \emph{Manin pair}  is a Lie algebra $\dd$, equipped with a nondegenerate $\ad$-invariant symmetric bilinear form (from now on referred to as \emph{metric}) $\l\cdot,\cdot\r$, together with a Lie subalgebra $\g\subset \dd$ which is \emph{Lagrangian}, i.e., maximal isotropic for the metric. Given a Lie group $G$ integrating $\g$, with an action on $\dd$ by metric-preserving automorphisms, extending the adjoint action on $\g$ and with infinitesimal action the adjoint action of $\g$ on $\dd$, we speak of a \emph{$G$-equivariant Manin pair}.  

\begin{definition}\label{def:quasipoisson}
A \emph{quasi-Poisson $\g$-manifold} for a Manin pair $(\dd,\g)$ is a manifold $M$ together with a Dirac morphism 
\[R_0\colon  (\T M,TM)\da (\dd,\g).\]
\end{definition}

\begin{remarks}
\begin{enumerate}
	\item In the language of \cite{bur:cou,igl:ham}, $M$ is a Hamiltonian space for the Manin pair $(\dd,\g)$. 
	\item The Dirac-geometric approach to quasi-Poisson manifolds was developed by Bursztyn and Crainic
	\cite{bur:di,bur:di1}. 
	The equivalence of the original definitions \cite{al:ma} with Definition \ref{def:quasipoisson} was proved by 
	Bursztyn, Iglesias-Ponte, and \v{S}evera \cite{bur:cou} (see also \cite{bur:qua}). 
\item The Dirac morphism determines a $\g$-action on $M$. If $(\dd,\g)$ is a $G$-equivariant Manin pair, 
and $R_0$ is equivariant for a $G$-action on $M$ integrating this $\g$-action, we say that $M$ is a \emph{quasi-Poisson $G$-manifold}. 
\item The original setting of \cite{al:ma} involves the choice of a Lagrangian subspace $\m\subset \dd$ complementary to $\g$. 
The backward image of $\m$ under $R_0$ is a Lagrangian subbundle complementary to $TM$, which therefore is the graph of a bivector field $\pi$. The properties of $R_0$ being a Dirac morphism then translate into the conditions 
	\[ [\pi,\pi]=\chi_M,\ \ \ \L_{\xi_M}\pi=-(F(\xi))_M,\]
where 	$\chi\in \wedge^3\g$ and  $F\in \wedge^2\g\otimes \g^*$ are defined in terms of $\m$, and 
$\wedge\g\to \Gamma(\wedge TM),\ \phi\mapsto \phi_M$ is the algebra morphism extending the action map.
The triple $(\g,\chi,F)$ defines a 
quasi-Lie-bialgebra.  If $\m$ is a Lagrangian Lie subalgebra, then $\pi$ is a Poisson structure, and the definition amounts to a 
Lie bialgebra action as in the work of Lu \cite{lu:mo}. 
\end{enumerate}	
\end{remarks}

The Courant morphism $R_0$ may be regarded as a Dirac structure 
\begin{equation}\label{eq:diractstructure}
 E\subset \dd\times \ol{\T M}\end{equation}
over $M$. This was first described as a Lie algebroid by Bursztyn-Crainic \cite[Theorem 2.5]{bur:di}, and recognized as a Dirac structure in \cite{bur:di1} (see also \cite[Section 6.2.3]{lib:th}). Since it is a Lie algebroid, the 
range of its anchor map defines a singular foliation of $M$. We shall call this the \emph{quasi-symplectic foliation}. 
\begin{remark}
After choice of a Lagrangian complement $\m$, defining a bivector field $\pi$, the quasi-symplectic foliation is spanned by 
vector fields
\[ \{\xi_M|\,\xi\in\g\}+\{\pi^\sharp(\mu)|\, \mu\in \Gamma(T^*M)\}.\] 
If $\m$ is a Lagrangian Lie subalgebra, so that $\pi$ is Poisson, one also has the symplectic foliation defined by the Poisson structure. The quasi-symplectic foliation need not coincide with the latter.
\end{remark}

\begin{definition}
A quasi-Poisson $\g$-space $M$ is called \emph{quasi-symplectic} if the Lie algebroid $E$ is transitive. 
\end{definition}

If $M$ is a quasi-Poisson $\g$-manifold, then every leaf $\iota\colon N\to M$ of its quasi-symplectic foliation becomes a quasi-Poisson $\g$-manifold, by restriction. The corresponding Courant morphism is the composition of $R_0$ with $\T\iota\colon \T N\da \T M$.

\begin{example}\label{ex:coisotropic}
Let $(\dd,\g)$ be a Manin pair, and $M$ a manifold with a $\dd$-action 
\[ \a\colon M\times \dd\to TM\] 
with co-isotropic stabilizers. Equivalently, the range of the dual map $\a^*\colon T^*M\to M\times \dd$ is isotropic. 
Then $M$ acquires the structure of a quasi-Poisson $\g$-manifold, with 
\[ R_0=\on{gr}(\a|_{M\times \g})+\on{gr}(\a^*)\subset \dd\times \ol{\T M}\]
given as the sum of graphs of the restricted action map 
$\a|_{M\times \g}\colon M\times \g\to TM$ and  the dual map $\a^*\colon T^*M\to M\times \dd$. (Note that the range of $\a^*$ is isotropic.) Its quasi-symplectic leaves are the $\g$-orbits. Given a Lagrangian complement $\m$ to $\g$ in $\dd$, and denoting by	$p\colon \dd\to \g$ be the projection along $\m$,  the resulting bivector field is given by $\pi^\sharp=-\a\circ p\circ \a^*$. (See \cite[Theorem 2.9]{al:pur}.) 
	
By the construction of \cite{lib:cou}, a $\dd$-action with coisotropic stabilizers also determines an \emph{action Courant algebroid}	$\AA_M=M\times \dd$ over $M$, with $E_M=M\times \g$ as a Dirac structure. A choice of a complementary Lagrangian subbundle $F_M\subset \AA_M$ (given for example by $M\times \m$) defines a bivector field $\pi\in \Gamma(\wedge^2 TM)$, which is Poisson whenever $F_M$ is integrable (e.g., 
$\m$ a Lie subalgebra). As shown in  \cite{lib:cou}, many well-known Poisson structures are obtained from this mechanism. 
\end{example}

\begin{example}\label{ex:amm0} \cite{al:qu}
Let $\g$ be a Lie algebra with an invariant metric, and denote by $\ol{\g}$ the same Lie algebra with the opposite metric. 
Consider the Manin pair $(\dd,\g)=(\ol{\g}\oplus \g,\g_\Delta)$ (where $\g_\Delta\cong \g$ is embedded diagonally).  
If $G$ is a Lie group integrating $\g$, the $\dd$-action on $G$ 
\[\dd\to \Gamma(TG),\ \ 
 (\xi_1,\xi_2)\mapsto -\xi_1^R+\xi_2^L\]
(where the superscripts indicate the left-invariant and right-invariant vector fields)  has Lagrangian stabilizers. 
It gives $G$ the structure of a quasi-Poisson $G$-manifold, with quasi-symplectic leaves the conjugacy classes.  
\end{example}

\subsection{Moment maps}
The  theory of Alekseev and Kosmann-Schwarzbach \cite{al:ma} considers moment maps valued in the homogeneous space $D/G$, where $(D,G)$ is a group pair 
integrating the given Manin pair $(\dd,\g)$. In the Dirac-geometric setting, it is natural to allow for more general 
targets. See \cite[Definition 3.1]{lib:sypo}.
\begin{definition}[Moment map targets]
	A \emph{moment map target} for a Manin pair $(\dd,\g)$ is a manifold $\Q$ together 
	with a transitive action of $\dd$ with Lagrangian stabilizers.	A  \emph{$G$-equivariant moment map target} for a $G$-equivariant Manin pair is a moment map target together 
	with a $G$-action integrating the $\g\subset \dd$-action. Given a moment map target $\Q$ for $(\dd,\g)$, we shall denote by $\Q^{\on{op}}$ the same space regarded as a moment map target for $(\ol{\dd},\g)$. 
\end{definition}
Moment map targets satisfy  $\dim \Q=\dim\g$. We obtain  a  Dirac structure
\[ (\AA_{\Q},E_{\Q})=(\Q\times \dd,\ \Q\times \g);\]
for $G$-equivariant moment map targets, this Dirac structure is $G$-equivariant. By dimension count,  the  Courant algebroid 
$\AA_{\Q}$ is \emph{exact}. 

\begin{definition}\label{def:qvalued}
For a quasi-Poisson $\g$-space $M$, with Dirac morphism $R_0$, we define a \emph{$\Q$-valued moment map} to be a  Dirac morphism 
\[ R\colon (\T M,TM)\da (\AA_{\Q},E_{\Q})\]
whose composition with $(\AA_{\Q},E_{\Q})\da (\dd,\g)$  equals $R_0$. For a quasi-Poisson $G$-space $M$, 
with a $G$-equivariant moment map target, we require $R$ to be $G$-equivariant. 
\end{definition}

When a quasi-Poisson $\g$-space admits a moment map, then the Dirac morphism $R$ is uniquely determined 
by $R_0$ together with $\Phi$. Indeed, $R\subset (\Q\times \dd)\times \ol{\T M}$, as a subbundle along $\on{gr}(\Phi)\subset \Q\times M$, agrees 
with $R_0\subset \dd \times \ol{\T M}$ under the identification $\on{gr}(\Phi)\cong M$. 

\begin{example}\label{ex:orbits}
As a special case of Example \ref{ex:coisotropic}, the moment map target $\Q$ is itself a quasi-Poisson $\g$-space, 
with quasi-symplectic foliation given by the $\g$-orbits.  The identity map  serves as a moment map. Similarly, the $\g$-orbits $\O\subset \Q$ are quasi-Poisson $\g$-spaces, with the  
inclusion map as a moment map. 
\end{example}

If $\m\subset \dd$ is a Lagrangian complement to $\g$, then the backward image of $\m$ under $R_0$ is a Lagrangian subbundle complementary to $TM$, hence  of the form  $\on{gr}(\pi_M)$ for a bivector field $\pi_M$ on $M$. 
In particular, we obtain such a bivector field on $\Q$ itself. The latter is given 
by (see, e.g., \cite[Section 3.2]{lib:cou})
\begin{equation}\label{eq:qbivector}
 \pi_\Q=\hh \sum_a (e_a)_\Q\wedge (f^a)_\Q\end{equation}
where $e_a$ is a basis of $\g$, and $f^a$ the dual basis of $\m$. Given a $\Q$-valued moment map, the Dirac morphism $R$ relates the Lagrangian splittings, and hence $\Phi$ relates the bivector fields: 
\[ \pi_M\sim_\Phi \pi_\Q.\]

\subsection{Examples}

\begin{example}
	[Classical moment maps]
	For any Lie algebra $\g$, one has the Manin pair 
	$(\dd,\g)$ where 
	$\dd=\g^*\rtimes \g$, with metric defined by the pairing between $\g$ and $\g^*$. The Lie algebra $\dd$ acts on 
	$\g^*$, with the first factor acting by translation and the second factor by the coadjoint representation. The stabilizers for this action are Lagrangian. This defines the standard moment map target 
	$\g^*$ for the classical theory of Kostant and Souriau \cite{so:st}. 
\end{example}
\begin{example}[$D/G$-valued moment maps] \label{ex:aks}
Suppose $(D,G)$ is a Lie group pair integrating the Manin pair $(\dd,\g)$. 	The $\dd$-action on $D/G$ 
(obtained by differentiating the natural  $D$-action) has Lagrangian stabilizers, given at $dG\in D/G$ by $\Ad_d(\g)$.  
These are the moment map targets from the Alekseev-Kosmann-Schwarzbach theory \cite{al:ma}. For $D=\g^*\rtimes G$, one recovers the classical theory. 
\end{example}

\begin{example}[$G$-valued moment maps] \label{ex:amm1}
Let $G$ be a Lie group whose Lie algebra $\g$ carries an invariant metric. By Example \ref{ex:amm0}, the $G\times G$-action on $G$ has coisotropic stabilizers; hence $\Q=G$ is a moment map target for the Manin pair 
	\[ (\dd,\g)=(\ol{\g}\oplus\g     ,\g_\Delta).\] 
This is a special case of Example \ref{ex:aks}, with the group pair $(D,G)=(G\times G,G_\Delta)$. 
Note that in this case, the anti-diagonal $\m=\{(\xi,-\xi)|\xi\in\g\}$ is an invariant Lagrangian complement, 
leading to a description of the theory in terms of bivector fields and 2-forms. This is the theory of 
 $G$-valued moment maps  
	\cite{al:qu,al:mom}. 
\end{example}

\begin{example}
Suppose $(\dd,\g)$ is a Manin pair, $D$ a Lie group integrating $\dd$, and 
$U\subset D$ any closed Lagrangian Lie subgroup. Then  $\Q=D/U$  defines a 
moment map target for $(\dd,\g)$ (just as for the $D/G$ example). For instance, 
 if $G$ is a semisimple real Lie group, with Iwasawa decomposition 
$G=KAN$, one may take $D=G\times G$ (with Lie algebra $\dd=\ol{\g}\oplus \g$) and 
$U=A_\Delta (N\times N)$. The resulting moment map target is $D/U=K\times K\times A$. 
	\end{example}

\begin{example}[$G^*$-valued moment maps]\label{ex:lu}
	According to Drinfeld's theory \cite{dr:qu}, multiplicative Poisson structures on a Lie group $G$ are classified by Manin triples $(\dd,\g,\h)$. That is, $\h\subset \dd$ is a   Lagrangian Lie subalgebra
	complementary to $\g$.  The metric on $\dd$ identifies $\h\cong \g^*$. Let $H=G^*$ be the 1-connected Lie group integrating $\h\cong \g^*$. The \emph{(right) dressing action} of $\dd$ on $\Q=G^*$ has generating vector fields $\zeta_H,\ \zeta\in\dd$ given 
	in terms of their contraction with the right-invariant Maurer-Cartan form by 
	\begin{equation}\label{eq:dressing} 
	\iota_{\zeta_H}\theta^R|_h=\pr_\h (\Ad_h\zeta)
	\end{equation}
	where $\pr_\h$ is projection to $\h$ along $\g$. See, e.g., \cite[Section 5.1]{lib:dir}.
	The stabilizer algebras for this action are given by 
	$\dd_h=\Ad_{h^{-1}}\g$, and in particular are Lagrangian. This defines the moment map target $G^*$ for 
	Lu's moment map theory \cite{lu:mo}. In certain examples, $G,\ G^*$ are realized as closed subgroups of $D$, 
	in such a way that the group multiplication $G\times G^*\to D$ is a global diffeomorphism. In this case, one has 
	$D/G\cong G^*$. 
\end{example}

\begin{example}[Dirac Lie groups]\label{ex:H} 
	Suppose $(\dd,\g)$ is a Manin pair such that 
	$\g$ admits a complement $\h$ that is a Lie subalgebra, but not necessarily Lagrangian. Let $H$ be a 1-connected integration of 
	$\h$.  The right dressing action 
	of $\dd$ on $\Q=H$ is defined by the same formula \eqref{eq:dressing} as before, and has 
	Lagrangian stabilizers. 
	One obtains on $H$ the structure of a \emph{Dirac Lie group} in the sense of \cite{lib:dir} (different from \cite{jot:dir,ort:mu}); 
	in the classification from \cite{lib:dir}, these are the \emph{exact} cases. 
	In particular, there is a Dirac morphism 
	\[ (\AA_H,E_H)\times (\AA_H,E_H)\da (\AA_H,E_H).\]
	with base map $\Mult_H$, and a Dirac morphism $(\AA_H,E_H)\da (\AA_H,E_H)^{\on{op}}$ with base map $\on{Inv}_H$.
\end{example}

\begin{example}[$P$-valued moment maps]\label{ex:semisimple}
	Let $\g^c$ be a complex semisimple Lie algebra, and $\g\subset \g^c$ a real form. 
	Let $\dd=\g^c$ as a 
	real Lie algebra, with metric $\l\cdot,\cdot\r$ given by the imaginary part of the Killing form. This vanishes on $\g$, and so 
	$(\dd,\g)=(\g^c,\g)$ is a Manin pair, resulting in a theory of  $G^c/G$-valued moment maps. The subspace
	$\m=i\g$ is an invariant Lagrangian complement.
	For the case that $\g$ is the Lie algebra of a compact Lie group $G$, we may use the Cartan decomposition $G^\C=GP$ 
	to identify the moment map target with $P=\exp(i\g)$. See \cite[Section 10.3]{al:mom}.
\end{example}

\section{The quasi-symplectic case}\label{sec:quasisymplectic}
The quasi-symplectic $\g$-spaces with $\Q$-valued moment maps correspond to the special case that the 
Dirac morphism $R\colon (\T M, TM)\da (\AA_\Q,E_\Q)$ is \emph{exact}. (See Lemma \ref{lem:exactly} below.) This allows for a description of such spaces
in terms of differential forms, once a \emph{splitting} of the Courant algebroid $\AA_\Q$ has been chosen.

 \subsection{Splittings}
 In Appendix \ref{app:A}, we recall how (isotropic) splittings $j\colon  TM\to \AA$ of exact Courant algebroids give an isomorphism $\AA\cong \T_\eta M$ for a closed 3-form $\eta$. The space of splittings is an affine space over the vector space of 
 2-forms; changing $j$ by a 2-form $\varpi$  (as in Equation \eqref{eq:changeofsplitting} below) replaces 
 $\eta$ with $\eta'=\eta+\d\varpi$.
 
 For a moment map target, a splitting $j$ of the Courant algebroid $\AA_\Q=\Q\times\dd$ 
may be regarded as a 1-form 
\begin{equation}
\alpha\in \Omega^1(\Q,\dd)
\end{equation}
with the properties 
\[ (\alpha(v)|_q)_\Q=v,\ \ \l\alpha(v),\alpha(v)\r=0,\] 
for all $v\in T_q\Q$. Modifying the splitting by a 2-form $\varpi$ replaces $\alpha$ with $\alpha^\varpi$, given by 
\begin{equation}\label{eq:alphavarpi} \l\alpha^\varpi,\zeta\r=\l\alpha,\zeta\r-\iota(\zeta_{\Q})\varpi.\end{equation} 
In terms of $\alpha$, 
the isomorphism 
\begin{equation}\label{eq:q}
s\colon \Q\times \dd\stackrel{\cong}{\lra} \T_\eta \Q\end{equation}
is explicitly given on constant sections by
\[ s(\zeta)=\zeta_{\Q}+\l\alpha,\zeta\r,\ \ \zeta\in\dd.\]
In particular, $\l s(\zeta_1),s(\zeta_2)\r=\l\zeta_1,\zeta_2\r$ and $\Cour{s(\zeta_1),s(\zeta_2)}=s([\zeta_1,\zeta_2])$.
One has the following formula for the 3-form $\eta$, obtained by the  same calculation as in the proof of 
Bursztyn-Crainic \cite[Lemma 3.5]{bur:di1} (for the case $\Q=D/G$):
 \begin{equation} \eta=-\hh \l \d\alpha,\alpha\r-\f{1}{6}\l [\alpha,\alpha],\alpha\r.\end{equation}


The Lie algebra $\dd$ acts on  $\AA_\Q=\Q\times \dd$ by infinitesimal Courant automorphisms, and hence acts on splittings. 
The 1-form part of the equation $\Cour{s(\zeta_1),s(\zeta_2)}=s([\zeta_1,\zeta_2])$ gives 
\begin{equation}\label{eq:interm}
\left\l  \L_{(\zeta_1)_{\Q}}\alpha+[\zeta_1,\alpha],\zeta_2\right\r=\iota_{(\zeta_2)_\Q}(\d\l\alpha,\zeta_1\r+\iota_{(\zeta_1)_\Q}\eta).
\end{equation}
From \eqref{eq:interm}, we see that the splitting is invariant under the action of a given element $\zeta\in\dd$ if and only if 
\begin{equation}\label{eq:equivariance} \d \l \alpha,\zeta\r+\iota_{\zeta_{\Q}}\eta=0.\end{equation}
If $\Q$ is a $G$-equivariant moment map target,  
the diagonal $G$-action on $\Q\times \dd$ is by Courant algebroid automorphisms. Hence, $G$ acts on the affine space of splittings; the 
pullback of the splitting defined by $\alpha$ under an element $g\in G$ is the splitting defined by 
 $g^{-1}\cdot\alpha=\Ad_{g^{-1}}\A_g^*\alpha$, where $\A_g\in \on{Diff}(\Q)$ is the action on the base. 
We hence obtain a map
\begin{equation}\label{eq:betaq}
 \beta\colon G\to  \Omega^2(\Q)\end{equation}
where $\beta(g)$ is the 2-form  (as in \eqref{eq:changeofsplitting}) 
relating the two splittings:
 \[ \l \A_g^*\alpha,\Ad_g\zeta\r-\l\alpha,\zeta\r=-\iota_{\zeta_{\Q}}\beta(g).\]
 Note that $\beta$ has  the cocycle property,  $\beta(hg)=\beta(g)+\A_g^*\beta(h)$, and that 
 \begin{equation}\label{eq:betaq1}
 \d\beta(g)=-\A_g^*\eta+\eta,\end{equation}
 as a special case of \eqref{eq:changeofsplitting}.  For later reference, we also note the 
 formula for the infinitesimal cocycle: 
\begin{lemma}\label{lem:hard} For all $\xi\in\g$, 
 \begin{equation}\label{eq:betaq2}\f{d}{d t}\Big|_{t=0}\beta(\exp(t\xi)) =
 \d\l\alpha,\xi\r+\iota_{\xi_{\Q}}\eta
 \end{equation}
\end{lemma}
\begin{proof}
Taking the derivative of \eqref{eq:betaq}, we obtain 
\[ \iota_{\zeta_{\Q}}\f{d}{d t}\big|_{t=0}\beta(\exp(t\xi))=\L_{\xi_{\Q}}\l\alpha,\zeta\r
-\l\alpha,[\xi,\zeta]\r.\]
Using the 1-form part of the identity $\Cour{s(\xi),s(\zeta)}=s([\xi,\zeta])$, 
%
we see that the right hand side may be written $\iota_{\zeta_{\Q}}(\d\l\alpha,\xi\r+\iota_{\xi_{\Q}}\eta)$. 
\end{proof}

\begin{example}\label{ex:dg}
The Courant algebroid over  $\Q=D/G$
does not have a distinguished splitting, in general. 
Suppose however that $\g\subset \dd$ admits a \emph{$G$-invariant} 
Lagrangian complement $\m$, as in Examples \ref{ex:amm1} and \ref{ex:semisimple}. 
(See \cite[Remark 5.2]{bur:di1} for a necessary and sufficient condition.) 
Then $\Ad_d(\m)$ depends only on the coset $dG$, and is a 
	complement to the stabilizer $\dd_{dG}=\Ad_d(\g)$. The corresponding 1-form  
	 $\alpha\in \Omega^1(D/G,\dd)$ is
\begin{equation}\label{eq:alpha1} 
p^*\alpha=-\Ad_d \pr_\m\theta^L_D,\end{equation}
 where $p\colon D\to D/G$ is the quotient map, $\theta^L_D$ is the left-invariant Maurer-Cartan form on $D$, and $\pr_\m$ denotes projection to $\m$ along $\g$. This splitting is 
 $D$-invariant; conversely, every $D$-invariant splitting is of this form.

\end{example}

\begin{example}
		In the 
	Poisson Lie group case $\Q=G^*$ (Example \ref{ex:lu}), the form
	\[ \alpha=\theta^L_{G^*}\]  defines a splitting, with $\eta=0$. This splitting is usually not $\g$-invariant (let alone $\dd$-invariant). 
\end{example}

\begin{example}\label{ex:H1}
	 Consider the Dirac Lie group case $\Q=H$ (Example \ref{ex:H}). 
	Let $\m$ be the Lagrangian complement given as the mid-point between $\h,\h^\perp$ in the affine space of complements to $\g$. 
	For $h\in H$, 
	the subspace $\Ad_{h^{-1}}\m$ is a Lagrangian complement to the stabilizer algebra $\dd_h=\Ad_{h^{-1}}\g$, and so defines a splitting of the Courant algebroid. 
	The corresponding 1-form is 
	\begin{equation}\label{eq:alpha2}
	\alpha=\Ad_{h^{-1}}\pr_\m\theta^R_H\in \Omega^1(H,\dd).\end{equation}

\end{example}

\begin{example}\label{ex:amm2}
	The Courant algebroid $\AA_G=G\times (\ol\g\oplus\g)$ over the 
	moment map target $\Q=G$ (Example \ref{ex:amm1}) is $G\times G$-equivariant for the action 
\begin{equation}\label{eq:ammaction}
 (a_1,a_2)\cdot (g,\xi_1,\xi_2)=(a_1 g a_2^{-1},\ \Ad_{a_1}\xi_1,\Ad_{a_2}\xi_2).\end{equation}
	The 1-form 
	\[ \alpha=\hh(-\theta^R_G,\theta^L_G)\in \Omega^1(G,\ol{\g}\times \g)\]
	defines a canonical $G\times G$-equivariant splitting, with $\eta=\f{1}{12}\theta^L\cdot [\theta^L,\theta^L]$
	the \emph{Cartan 3-form} on $G$. This splitting may be regarded as a special case of \eqref{eq:alpha1} (viewing $G$ as a quotient $D/G$ for $D=G\times G$) 
	or as a special case of \eqref{eq:alpha2} (for the choice $\h=0\oplus \g\subset \dd$). 
\end{example}

\subsection{The quasi-symplectic case}
Recall from Appendix \ref{app:A} the notion of \emph{exact} (or \emph{full}) Courant morphisms between exact Courant algebroids. Given 
splittings, these may be described in terms of 2-forms. 

\begin{lemma}\label{lem:exactly}
	For a quasi-Poisson $\g$-space $M$ with ${\Q}$-valued moment map $\Phi$, the associated Dirac morphism 
	$R\colon (\T M,TM)\da (\AA_{\Q},E_{\Q})$ is exact if and only if $M$ is quasi-symplectic. 
\end{lemma}
\begin{proof}
	The Dirac morphism $R$ is exact if and only if the anchor map restricts to a surjection $R\to \on{gr}(T\Phi)$. 
	In terms of the identification $\on{gr}(T\Phi)\cong TM$, the range of the anchor map of $R$ 
	equals that of  $R_0\colon \T M\da \dd$. By definition, $M$ is quasi-symplectic 
	if and only if the latter map is surjective. 
\end{proof}

By spelling out the description  of exact Dirac morphisms, see Appendix \ref{app:A}, Equation \eqref{eq:exactmor}, we obtain: 

\begin{proposition}\label{prop:qvaluedmomentmaps}
	Given a splitting $\alpha\in \Omega^1(\Q,\dd)$, with corresponding 3-form $\eta\in \Omega^3(\Q)$,  a quasi-symplectic $\g$-space $M$ with ${\Q}$-valued moment map
	is described by a 
	2-form $\omega\in \Omega^2(M)$ and a $\g$-equivariant map $\Phi\colon M\to {\Q}$ satisfying the properties 
	\begin{itemize}
		\item[(i)] $\d\omega=-\Phi^*\eta$,
		\item[(ii)] $\iota_{\xi_M}\omega=-\Phi^*\l\alpha,\xi\r,\ \ \xi\in\g$,
		\item[(iii)] $\ker(T\Phi)\cap \ker(\omega)=0$.
	\end{itemize} 
\end{proposition}

\begin{remarks}\label{rem:div}
	\begin{enumerate}
		\item Changing the splitting by $\varpi$ replaces $\omega$ with $\omega+\Phi^*\varpi$. 
		\item As a consequence of (i),(ii), the 2-form $\omega$ satisfies 
		\[ \L_{\xi_M}\omega=-\Phi^*\big( \d\l\alpha,\xi\r+\iota_{\xi_{\Q}}\eta \big)\]
		for $\xi\in\g$. 
		The right-hand side vanishes precisely when the splitting is $\g$-invariant. See Equation \eqref{eq:equivariance}.
		
		\item\label{it:G} If the $\g$-action on $M$ integrates to a $G$-action, one describes a 
		\emph{quasi-symplectic $G$-space with $\Q$-valued moment map}, by the additional requirement that
		\[ \A_g^*\omega-\omega=\Phi^*\beta(g)\] 
		with the cocycle \eqref{eq:betaq}.
		This property is automatic if $G$ is connected.
	\end{enumerate}	
\end{remarks}

\begin{lemma}
	Condition (iii) may be replaced with
	\[ (iii)'\ \ \ker(\omega_m)=\{\xi_M(m)|\ \xi\in \g,\ \l\alpha_{\Phi(m)},\xi\r=0\},\ \ m\in M.\]
\end{lemma}
\begin{proof}	Let $q=\Phi(m)$, with stabilizer $\dd_q$. 
	The range of $\alpha|_q$ is a complement to $\dd_q$, and hence is non-singularly paired with $\dd_q$. 	It follows that the map 
	\[ \dd_q\to T^*_q\Q,\ \zeta\mapsto \l\alpha|_q,\zeta\r\] 
	is an isomorphism. 
	Using this observation, we see that 
	Condition (iii)' implies (iii). Conversely, suppose (iii) holds, and let $v\in \ker(\omega_m)$. 
	Let 
	\[ w=T_m\Phi(v),\ \ \xi=\iota_w \alpha_{q}.\]
	The moment map condition (ii) shows that for all 
	$\xi'\in \g$, 
	\[ \l\xi,\xi'\r=\l\iota_w \alpha_{q},\xi'\r=
	\omega(\xi'_M|_m,v)=
	0.\] 
	Hence, $\xi\in\g$. 
	We have $\l\alpha_q,\xi\r=0$, since $\xi\in \on{ran}(\alpha_{q})=\on{ran}(\alpha_{q})^\perp$. 
	Using (ii) again, it follows that 
	$\xi_M|_m\in \ker(\omega_m)$. 
	The difference $v-\xi_M|_m$ lies in $\ker(\omega_m)\cap \ker(T_m\Phi)$, and is hence equal to $0$. This proves (iii)'. 
\end{proof}
As a direct consequence, $\omega$ is nondegenerate at all those points $m\in M$ for which the range of 
$\alpha_{\Phi(m)}$ is a complement to $\g$. Since $\alpha_{\Phi(m)}\colon T_m\Q\to \dd$ defines a splitting of the $\dd$-action, this holds in particular at all points where the stabilizer $\dd_{\Phi(m)}$ is equal to $\g$.


\subsection{Symplectic reduction}
Suppose $M$ is a quasi-symplectic $G$-space with 
moment map $\Phi\colon M\to \Q$, with the associated $G$-equivariant exact  
Dirac morphism 
\[ R\colon (\T M,TM)\da (\AA_{\Q},E_{\Q}).\] 	
(The choice of splitting defines a 2-form $\omega$, 
but the following results do not depend on such a choice.) 

\begin{proposition}\label{prop:locfree}
	On the pre-image of the set of regular values of $\Phi$, the $G$-action is locally free. 
\end{proposition}
\begin{proof}
	Let $m\in M$ be such that $T\Phi|_m$ is surjective, and let $\xi\in \g_m$ (the infinitesimal stabilizer of $m$). 
	Then $0_m\sim_R \xi$. By surjectivity  there exists, for every given  $\zeta\in\dd$, some $v\in TM|_m$ with 
	$T\Phi|_m(v)=\zeta_{\Q}|_{\Phi(m)}$. Since $R$ is an exact Courant morphism, the pair 
	$(\zeta_{\Q}|_{\Phi(m)},v)\in \on{gr}(T\Phi)$ 
	lifts to an element 
	\[ (\zeta',v+\mu)\in R\]
	where $\mu\in T^*M|_m$ and with $\zeta-\zeta'\in \dd_{\Phi(m)}$. Since $R$ is Lagrangian, the pairing of this element with $(\xi,0)\in R$ is zero. This gives 
	$\l\zeta',\xi\r=0$. But also $\l \zeta-\zeta',\xi\r=0$, since $\zeta-\zeta',\xi$ are both in the Lagrangian subspace $\dd_{\Phi(m)}$. 
	We conclude $\l \zeta,\xi\r=0$ for all $\zeta\in\dd$, and hence  $\xi=0$. 
\end{proof}

Let $\lambda\in\Q$ be a regular value of $\Phi$, and $G_\lambda$ its stabilizer. 
By Proposition \ref{prop:locfree}, the $G_\lambda$-action on $\Phinv(\lambda)$ is locally free. If the action is free and proper, the quotient is a manifold. 	

\begin{proposition}[Symplectic reduction] Suppose $\lambda$ is a regular value of $\Phi\colon M\to \Q$, and that the 
	$G_\lambda$-action on $\Phinv(\lambda)$ is free and proper. Then 
	\[ M\qu_{\!\lambda}\, G=\Phinv(\lambda)/G_\lambda\] 
	acquires a symplectic 2-form.	
\end{proposition}	
\begin{proof}
	Let $Z=\Phinv(\lambda)$. 
	Denote by 
	$i_Z\colon Z\to M$ and $i_\lambda\colon \{\lambda\}\to \Q$ the inclusions. 
	The pullback Courant algebroids are 
	$i_Z^!\T M=\T Z$ and $i_\lambda^!\AA_\Q=0$, and $R$ defines a $G_\lambda$-equivariant exact Courant morphism 
	\[ R_Z\colon \T Z\da 0.\] 
	Explicitly, for  $v+\nu\in \T Z$ we have 
	\begin{equation}\label{eq:explicit1} 
	v+\nu\sim_{R_Z} 0 \Leftrightarrow \exists \mu\in T^*M|_Z,\ \zeta\in \dd_\lambda\colon \nu=i_Z^*\mu,\ v+\mu\sim_R\zeta.\end{equation}
	We claim that the intersection $\ker(R_Z)\cap TZ$ consists exactly of the $G_\lambda$-orbit directions. By \eqref{eq:explicit1}, a tangent vector $v\in TZ$ satisfies $v\sim_{R_Z}0$ if and only if there exists $\xi\in \dd_\lambda$ 
	with $v\sim_R\xi$. Since $R$ is a Dirac morphism, the forward image of $TM|_m$ under $R$ equals $\g$; hence 
	$\xi\in \dd_\lambda\cap \g=\g_\lambda$, and by the uniqueness condition for Dirac morphisms  we have 
	$v=\xi_M|_m$. This proves the claim. 
	Since $\T Z$ has a distinguished splitting, the Courant morphism  
	$R_Z$ is equivalent to a $G_\lambda$-invariant, closed 2-form $\omega_Z\in \Omega^2(Z)$, with kernel $\ker(\omega_Z)=
	\ker(R_Z)\cap TZ$. It descends to a symplectic 2-form on $Z/G_\lambda=M\qu_{\!\lambda}\, G$.
\end{proof}

\begin{remark} The assumption that 
	$\lambda$ is a regular value of the moment map is equivalent to $\Phi$ being transverse to the orbit $\O=G.\lambda$. The reduced space 
	$M\qu_{\!\lambda}\, G$ is canonically identified with 
	\[ M\qu_{\!\O}\, G=\Phinv(\O)/G.\]
\end{remark}
\begin{remark}
	The proposition (and its proof) generalizes, as follows: 	
	Suppose $(\mathbb{B},F)$ is a Dirac structure over another manifold $N$, and 
	\[ R\colon (\T M,TM)\da (\AA_{\Q},E_{\Q})\times (\mathbb{B},F)\]
	is a $G$-equivariant Dirac morphism with base map $(\Phi,\Psi)\colon M\to \Q\times N$ (where the $G$-action on $\mathbb{B}\to N$
	is trivial). By the same argument as above, the action on the pre-image of the set of regular values of $\Phi$ is locally free. If $\lambda$ is a regular value of $\Phi$, and the action of $G_\lambda$ on $\Phinv(\lambda)$ is free and proper, 
	then $R$ descends to a Dirac morphism 
	\[ R\qu_{\!\lambda}\, G\colon (\T (M\qu_\lambda G),\ T (M\qu_{\!\lambda}\, G))\da (\mathbb{B},F).\]
	For instance, given two Manin pairs $(\dd_i,\g_i)$ with moment map targets $\Q_i$, and  
	given a quasi-symplectic $G=G_1\times G_2$-space with $\Q=\Q_1\times \Q_2$-valued moment map, one may reduce with respect to just the $G_1$-action and obtain a quasi-symplectic $G_2$-space. 
\end{remark}

\section{Integration of moment map targets}\label{sec:integration}

\subsection{Quasi-symplectic groupoids}\label{subsec:groupoid}
The notion of a quasi-symplectic group\-oid integrating an exact Dirac structure was introduced in \cite{bur:int} and 
\cite{xu:mor}, independently, using splittings of the Courant algebroid. (Reference \cite{bur:int} uses the terminology of \emph{pre-symplectic groupoid}.) 
The following formulation in terms of Dirac morphisms may be found in \cite{bur:cou}; see also 
\cite[Remark 6.2.1]{lib:th} and \cite{igl:uni}.

Let $(\AA,E)$ be a Dirac structure (not necessarily exact) over a manifold $M$. 
A \emph{quasi-Poisson groupoid} integrating this Dirac structure is a Lie groupoid $\G\rra M$ integrating the Lie algebroid $E\Rightarrow M$, together with a multiplicative Dirac morphism 
\begin{equation}\label{eq:K}
 K\colon (\T\G,T\G)\da (\AA,E)\times (\AA,E)^{\on{op}},\end{equation}
with base map $(\tz,\sz)\colon \G\to M\times M$. Here, multiplicativity means that 
$K\subset \AA\times \ol{\AA}\times \ol{\T\G}$ is a subgroupoid along the graph of the base map $(\tz,\sz)$, for the standard groupoid structure $\T\G\rra TM\oplus \on{Lie}(\G)^*$ and 
the pair groupoid structure $\AA\times \ol{\AA}\rra \AA$. We speak of a \emph{quasi-symplectic groupoid} if $\AA$ and the 
morphism $K$ are exact. 

\begin{example}
Let $(\dd,\g)$ be a $G$-equivariant Manin pair, regarded as a  Dirac structure over $M=\pt$.
A quasi-Poisson groupoid integrating this Dirac structure is given by the morphism $K\colon \T G\da (\dd,\g)\times (\dd,\g)^{\on{op}}$, 
where 
$K\subset \dd\times\ol{\dd}\times \ol{\TG}$ 
is spanned by all $(\xi,\xi',(\xi')^L-\xi^R)$ with $\xi,\xi'\in \g$, together with all $(\Ad_g\zeta,\zeta,-\theta^L_G\cdot\zeta)$ with $\zeta\in\dd$.   
\end{example}

We are mainly interested in the special case that the Courant algebroid $\AA$ and the morphism $K$ are \emph{exact}.  
Given a splitting of $\AA$, with associated 3-form $\eta$, such a morphism is described by a 
\emph{multiplicative} 2-form $\omega\in \Omega^2(\G)$ with $\d\omega=\sz^*\eta-\tz^*\eta$, 
satisfying the minimal degeneracy condition 
\begin{equation}\label{eq:minimaldeg}
 \ker(\omega)\cap\ker(T\tz)\cap\ker(T\sz)=0\end{equation}
and a moment map property.  The theory in  \cite{bur:int} shows that whenever the Lie algebroid $E$ admits a source-simply connected integration $\G$, then the latter has a canonical structure as quasi-symplectic groupoid. 

\subsection{Integration of moment map targets $\Q$}
Let $\Q$ be a moment map target for the Manin pair $(\dd,\g)$. 
Given a splitting $\alpha\in \Omega^1(\Q,\dd)$, let 
\begin{equation}\label{eq:chi}
 \chi\in C^\infty(\Q,\dd^*\otimes \dd^*),\ \ 
\chi(\zeta_1,\zeta_2)=\l\iota_{(\zeta_1)_{\Q}}\alpha,\zeta_2\r.\end{equation}
This satisfies 
\[ \chi(\zeta_1,\zeta_2)+\chi(\zeta_2,\zeta_1)=\l\zeta_1,\zeta_2\r;\]
in particular, the restriction to a bilinear form on $\g$ is skew-symmetric. 
Let 
$\beta$ be the cocycle \eqref{eq:betaq} regarded as a 2-form of bidegree $(0,2)$ on $G\times \Q$.
We denote by 
 $\A_g\in \on{Diff}(\Q)$ the action of $g\in G$; pullbacks of differential forms under $\A_g$ may be regarded as families of 
 forms on $\Q$ (depending on 
 $g\in G$ as a parameter), hence as forms on $G\times \ca{Q}$.

 %
\begin{theorem}\label{th:groupoid}
Let $\Q$ be a $G$-equivariant moment map target for the Manin pair $(\dd,\g)$. Then the action groupoid 
\[ \G_{\Q}=G\ltimes \Q\rra \Q\] 
is canonically a quasi-symplectic groupoid integrating the Dirac structure $(\AA_{\Q},E_{\Q})$.  For a given splitting $\alpha\in \Omega^1(\Q,\dd)$, 
the quasi-symplectic 2-form on $\G_Q$ is given by  
\[ \omega=\A_g^*\left(-\l \alpha,\theta^R\r+\hh \chi(\theta^R,\theta^R)\right)+\beta.\]
\end{theorem}
For the case that $G$ is 1-connected and the splitting is $G$-equivariant, this is proved in  
\cite[Proposition 6.10]{bur:int}. The subsequent \cite[Remark 6.11]{bur:int} explains how to generalize the 
formula to non-equivariant splittings. For the convenience of the reader, and since we do not impose any connectivity assumptions on $G$, we provide a  self-contained proof  in Appendix \ref{app:C}.

\section{Quotients and lifts of quasi-Poisson spaces}\label{sec:quotientlifts}

\subsection{Quotients of moment map targets}\label{subsec:quotients}
Let $(\dd,\g)$ be a Manin pair. Suppose 
 $\cc\subset \dd$ is a Lie subalgebra containing $\g$. Then the orthogonal space 
$\k=\cc^\perp$ is an ideal in $\cc$, hence also in $\g$, and 
we obtain a reduced Manin pair 
\[ (\dd',\g')=(\cc/\k,\g/\k).\] 

For reduction of \emph{$G$-equivariant} Manin pairs $(\dd,\g)$, we assume that $\k$ integrates to a closed normal 
 subgroup $K\subset G$. (Normality is automatic if $K$ is connected.) Then 
 $\k$ and $\cc=\k^\perp$ are   
$\Ad_G$-invariant, and hence $(\dd',\g')$ is a $G'=G/K$-equivariant Manin pair. 
Suppose $\Q$ is a $G$-equivariant moment map target for $(\dd,\g)$, 
and that the action of $K\subset G$ on $\Q$ is free and proper.
Then the $\cc\subset \dd$-action on $\Q$ descends to a $\dd'$-action on 
the quotient 
\[ \Q'=\Q/K.\]
\begin{lemma}
The $\dd'$-action on $\Q'$ is transitive, with Lagrangian stabilizers. 
\end{lemma}
\begin{proof}
By definition, the $\dd$-action on $\Q$ is transitive, with Lagrangian stabilizers. Since the $K$-action is free, 
we have $\dd_q\cap \k=0$ for all $q\in \Q$. Taking orthogonals, this shows that $\dd_q+\cc=\dd$, so that already the 
action of $\cc$ is transitive. It follows that the induced $\dd'$-action on $\Q'$ is transitive. 
The stabilizer $\dd'_{qK}$ contains the Lagrangian subspace 
$(\dd_q\cap\cc)/(\dd_q\cap \k)$, and in fact has to be equal for dimension reasons. 
\end{proof}

In conclusion, $\ca{Q}'$ becomes a $G'=G/K$-equivariant moment map target 
for $(\dd',\g')$. 
As before, we denote 
\[ (\AA_{\Q},E_{\Q})=(\Q\times\dd,\Q\times \g)\] 
and similarly for $\Q'$.

\begin{proposition}\label{prop:targetreduction}
	The Dirac structure $(\AA_{\Q'},E_{\Q'})$ is the reduction of  $(\AA_\Q,E_\Q)$ 
	under the action of $K$.
\end{proposition} 
\begin{proof}
	The $K$-action on $\AA_\Q$  has isotropic generators  
	$\varrho\colon \k\to \Gamma(\AA_\Q)$
	given by the constant sections, $\xi\mapsto \xi$. Thus, 
	$\on{ran}(\varrho)^\perp=\Q\times \cc$, and the reduced Courant algebroid is 
	\[  \f{\on{ran}(\varrho)^\perp}{\on{ran}(\varrho)}\Big/K=\f{\Q\times \cc}{\Q\times \k}\Big/K=\Q'\times \dd'=\AA_{\Q'}.\] 
	The constant sections of $\AA_{\Q'}$ defined by elements 
	$\zeta'\in\dd'$ are the images of the constant sections 
	of $\AA_\Q$ defined by any choice of lift $\zeta$; hence the Courant bracket of constant sections is simply the given Lie bracket on $\dd'$.  We have $\on{ran}(\varrho)\subset E_\Q$, hence 
	the image of  $E_\Q\subset \on{ran}(\varrho)^\perp$ under the reduction is $E_{\Q'}$. 
\end{proof}

The choice of a principal connection  gives a method for descending splittings.  

\begin{proposition}
	Suppose $\alpha\in \Omega^1(\Q,\dd)$ is a $K$-invariant splitting of $\Q\times \dd$, and $\theta\in \Omega^1(\Q,\k)$ a principal connection. Let $\alpha^\varpi$ be the twist of $\alpha$ by the 2-form 
	\[ \varpi=-\l\alpha,\theta\r+\hh \chi(\theta,\theta)\]
	where $\chi$ is defined in \eqref{eq:chi}. 
	Then $\alpha^\varpi$ takes values in $\cc$, and its projection $\pr_{\dd'}\alpha^\varpi$ descends to a splitting $\alpha'\in
	\Omega^1(\Q',\dd')$. 
\end{proposition}
\begin{proof}Contraction of $\varpi$ with $\zeta_\Q$ for 
	$\zeta\in\dd$ gives 
	\[ \iota_{\zeta_{\Q}}\varpi=\l\alpha,\iota_{\zeta_{\Q}}\theta\r-\chi(\zeta-\iota_{\zeta_{\Q}}\theta,\theta).\]
	By definition of the  twisted splitting (see \eqref{eq:alphavarpi}), this gives 
	\[
	\l\alpha^\varpi,\zeta\r=\l\alpha,\zeta-\iota_{\zeta_{\Q}}\theta\r+\chi(\zeta-\iota_{\zeta_{\Q}}\theta,\theta)
	\]
	Note that this expression vanishes if $\zeta\in\k$; hence $\alpha^\varpi$ takes values in $\cc=\k^\perp$ as required. 
	Since $\alpha^\varpi$ is $K$-equivariant, the $\dd'$-values 1-form $\pr_{\dd'}\alpha^\varpi$ is $K$-invariant. To show that it descends, we must show that it is $K$-basic. Thus let $\xi\in \k$. 
	For $\zeta\in \cc$,  we see 
	\begin{align*} 
	\iota_{\xi_{\Q}}\l\alpha^\varpi,\zeta\r&= \l\iota_{\xi_{\Q}}\alpha,\zeta\r-\l\iota_{\xi_{\Q}}\alpha,\iota_{\zeta_{\Q}}\theta\r
	+\chi(\zeta,\xi)-\chi(\iota_{\zeta_{\Q}}\theta,\xi)
	\\&=
	\chi(\xi,\zeta)+\chi(\zeta,\xi)-\chi(\xi,\iota_{\zeta_{\Q}}\theta)-\chi(\iota_{\zeta_{\Q}}\theta,\xi)
		\\
	&=\l\zeta,\xi\r-\l \iota_{\zeta_{\Q}}\theta,\xi\r=0.
	\end{align*}
 It follows that $\iota_{\xi_{\Q}}\alpha^\varpi$ takes values in $\cc^\perp=\k$, and so 
  $\iota_{\xi_{\Q}}\pr_{\dd'}\alpha^\varpi=0$. 
\end{proof}

\begin{remark}[Bivector fields] \label{rem:bivac1}
	Suppose $\m$ is a Lagrangian subspace complementary to $\g$ in $\dd$, defining the bivector field $\pi_\Q$. 
	The image $\m'$ of 
	$\m\cap\cc$ in $\dd'$ is a Lagrangian complement to $\g'$, defining a bivector field $\pi_{\Q'}$.  
	Since the reduction morphism 
	$\Q\times \dd\da \Q'\times \dd'$ (where $(q,\zeta)$ is related to $(q',\zeta')$ if and only $\zeta\in\cc$, and 
	$q',\zeta'$ are the images of $q,\zeta$, respectively)  
	relates the Lagrangian splittings (in the sense of \cite[Section 3.5]{lib:cou}), 
	it follows that the quotient map $\Q\to \Q'$ takes $\pi_\Q$ to $\pi_\Q'$. Alternatively, this may be deduced from the formula \eqref{eq:qbivector}.
\end{remark}

\subsection{Lifting theorem}
Let $\Q'=\Q/K$ as above. If $M$ is a manifold with a $G$-equivariant map $\Phi\colon M\to \Q$, then the action of $K$ on $M$ is again free and proper (by equivariance of $\Phi$), and hence $M'=M/K$ with $G'$-equivariant map $\Phi'\colon M'\to \Q'$ is defined. Conversely, given 
a  $G'$-equivariant map $\Phi'\colon M'\to \Q'$, we recover $M$ by pulling back the principal $K$-bundle 
$\Q\to \Q'$. It is possible to incorporate Dirac geometry in this construction:

\begin{theorem}[Lifting theorem]	\label{th:liftingtheorem}
	Let $\Q'=\Q/K$ and $G'=G/K$ as above. 
	There is a 1-1 correspondence between 
	\begin{itemize}
		\item quasi-Poisson $G$-manifolds $M$ with $\Q$-valued moment maps, 
		\item quasi-Poisson $G'$-manifolds $M'$ with $\Q'$-valued moment maps.
	\end{itemize}	
	Furthermore, $M'=M/K$ is quasi-symplectic if and only if $M$ is quasi-symplectic. 
\end{theorem}
\begin{proof}
	The theorem is a straightforward  application of a general result on reductions of Dirac morphisms, see \cite{cab:dir} as well as 
Proposition \ref{prop:part2} in the appendix. By definition, the structure of a quasi-Poisson $G$-manifold $M$ with 
	$\Q$-valued moment map $\Phi$ is described by a 
	 $G$-equivariant Dirac morphism
	\[  R\colon (\T M,TM)\da \big(\AA_\Q,\ E_\Q\big).\]
	This Dirac morphism intertwines the generators for the $\g$-action, given for $\T M$ by the generating vector fields 
	for the action on $M$ and for $\AA_\Q$ by the constant sections of $\AA_\Q$:
	\[ (\xi_M,0)\sim_{{R}} \varrho(\xi).\]
	In particular, $R$ intertwines the generators for the $K$-action. Since $\T M'$ is the reduction by $K$ of the Courant algebroid 
	$\T M$, Proposition \ref{prop:part2} gives a 1-1 correspondence between $G$-equivariant Dirac morphisms $R$, as above, and 
	$G'=G/K$-equivariant Dirac morphism 
	\[ R'  \colon (\T M',TM')\da \big(\AA_{\Q'},\ E_{\Q'}\big).\]
Moreover, $R$ is exact if and only if $R'$ is exact. 
\end{proof}

\begin{example}
Let $\G=
G\ltimes \Q\rra \Q$ be the quasi-symplectic groupoid integrating the Dirac structure $(\AA_\Q,E_\Q)$. Recall that $\G$ is a
quasi-symplectic $G\times G$-manifold, with $\Q\times \Q^\op$-valued moment map given by $(\tz,\sz)$. Taking the quotient under the 
action of $K\times K$, we obtain the quasi-symplectic groupoid $\G'=G'\ltimes \Q'\rra \Q'$ integrating $\Q'$. We may also take a
quotient with respect to only the second $K$ factor. The resulting space $N=\G/(\{e\}\times K)$ gives a 
Morita equivalence of quasi-symplectic groupoids, 
\begin{equation}\label{eq:corr}	\begin{tikzcd}
[column sep={7em,between origins},
row sep={4.5em,between origins},]
\G \arrow[d,shift left]\arrow[d,shift right]
& N \arrow[dl, "p"] \arrow [dr, "q"']   & \G' \arrow[d,shift left]\arrow[d,shift right] \\
\Q & & \Q'
\end{tikzcd}
\end{equation}
\end{example}

\begin{remark}[Bivector fields] \label{rem:bivac2}
Suppose $\m$ is a Lagrangian subspace complementary to $\g$ in $\dd$, and let $\m'\subset \dd'$ its reduction as as in Remark \ref{rem:bivac1}.
The backward image of $\Q\times \m$ under $R$ (equivalently, of $\m$ under $R_0$) is a Lagrangian complement to $TM$, hence is the graph of a bivector field $\pi_{M}$, 
in such a way that $\pi_M\sim_\Phi \pi_\Q$. Similarly, we obtain a bivector field $\pi_{M'}$. 
By Remark \ref{rem:bivacA}, we have $\pi_{M'}\sim_{\Phi'}\pi_{\Q'}$.  
\end{remark}


\subsection{Equivalence between $D$-valued moment maps and $D/G$-valued moment maps}
Our main example for the quotient construction is the following. Let $(D,G)$ be a group pair integrating the Manin 
pair $(\dd,\g)$, and $D/G$ the moment map target with the Dirac structure 
\begin{equation}\label{eq:dg}
  (\AA_{D/G},E_{D/G})=(D/G\times \dd,\,D/G\times \g).\end{equation}
We may regard $(\dd,\g)$ as a reduction of the Manin pair $(\dd\oplus \ol{\dd},\g\oplus \g)$ with respect to 
$\cc=\dd\oplus \g$, i.e. $\k=\cc^\perp=0\oplus \g$. Let
\begin{equation}\label{eq:d}
 (\AA_D,L_D)=(D\times (\ol{\dd}\oplus \dd),\,D\times (\g\oplus \g)),\end{equation}
where $\AA_D$ is as in Example \ref{ex:amm1} (with $D$ playing the role of $G$); the notation $L_D$ is 
used to avoid confusion with the Dirac structure $E_D$ (defined by $\dd_\Delta$). Note that this uses 
$\ol{\dd}\oplus \dd$ rather than $\dd\oplus \ol{\dd}$; the sign change of the metric means using the opposite 
Dirac structure $(\AA_D,L_D)^{\on{op}}=(\ol{\AA},L_D)$.

By Proposition \ref{prop:targetreduction}, 
the Dirac structure \eqref{eq:dg} is the reduction of $(\AA_D,L_D)^\op$ with respect to $K=\{e\}\times G$. 
Alternatively, we may use the isomorphism of Dirac structures 
\[ (\AA_D,L_D)\da (\AA_D,L_D)^{\on{op}}\]
given by $(d,\zeta,\zeta')=(d^{-1},\zeta',\zeta)$. Thus, \eqref{eq:dg} may also be regarded as the reduction of \eqref{eq:d} under 
$K=G\times \{e\}$; the quotient map reads $d\mapsto d^{-1}G$.

 Since $\AA_D$ has a canonical $D\times D$-equivariant splitting given by 
$\alpha=\hh(-\theta^R_D,\theta^L_D)$
(cf.~ Example \ref{ex:amm2}), we have 
\begin{equation}\label{eq:splittingofAD} \AA_D\cong \T_\eta D\end{equation}
with the Cartan 3-form $\eta=\f{1}{12}\l\theta^L_D,[\theta^L_D,\theta^L_D]\r$.  The action groupoid for the $G\times G$-action is the corresponding quasi-symplectic groupoid integrating this Dirac structure. In terms of the canonical splitting of $\AA_D$, the 2-form on 
$\G=(G\times G)\ltimes D$ is explicitly given by 
\begin{equation}\label{eq:explicitomega}
 \omega=\hh \left(\l d^*\theta^L,h_2^*\theta^L\r+\l d^*\theta^R,h_1^*\theta^L\r+\l h_1^*\theta^L,\Ad_d h_2^*\theta^L\r\right).
\end{equation}
The Lifting Theorem \ref{th:liftingtheorem} establishes a 1-1 correspondence between quasi-Poisson $G$-spaces $M$ 
with $D/G$-valued moment map $\Phi\colon M\to D/G$ and quasi-Poisson $G\times G$-spaces with $D$-valued moment map 
$\wh{\Phi} \colon \wh{M}\to D^\op$, where we use  the notation $D^\op$ to denote the moment map target $D$ with the 
Dirac structure $(\AA_D,L_D)^\op$. Alternatively, after  interchanging the two $G$-factors and composing $\wh{\Phi}$ with inversion, it is equivalent to  quasi-Poisson $G\times G$-spaces with $D$-valued moment maps.

\begin{example}
By Example \ref{ex:orbits}, the moment map target $D/G$ is itself a quasi-Poisson $G$-manifold, 
with moment map the identity. The lifted space is $D^\op$ as a quasi-Poisson $G\times G$-manifold, with moment 
map the identity. If $\O=G.(dG)\subset D/G$ is an orbit for the $G$-action, then  $\wh{\O}$ is the double coset
$GdG\subset D$. 	In particular, the base point $\O=\{eG\}\subset D/G$ corresponds to $G\subset D$. 
\end{example}

\begin{example}
By applying the isomorphism $D\cong D^\op$ to the second moment map component of the quasi-symplectic groupoid $(G\times G)\ltimes D\rra D$, we produce a quasi-symplectic $(G\times G)\times (G\times G)$-space with $D^\op\times D^\op$-valued moment map. The subsequent quotient by the action of 
$(\{e\}\times G)\times (\{e\}\times G)$ is just $D$, using the quotient map $(g_1,g_2,d)\mapsto g_1d$. The map $(\tz,\sz)$ descends to the map 
\[ D\to D/G\times D/G,\ d\mapsto (dG,\ d^{-1}G).\] 
This example of a quasi-symplectic $G\times G$-space with $D/G\times D/G$-valued moment map is due to Bursztyn-Crainic, see \cite[Theorem 4.1]{bur:di1}. 
\end{example}

The reformulation of $D/G$-valued moment maps as $D$-valued moment maps comes with several benefits: 
\begin{enumerate}
	\item There are simple constructions of  \emph{fusion products} and \emph{conjugates}.
	\item The Courant algebroid $\AA_D$ has a canonical splitting, allowing us to write the axioms for the `quasi-symplectic case' directly in terms of differential forms.  
	\item One obtains  new examples, coming from moduli spaces of flat $D$-bundles. 
\end{enumerate} 
We will explore these aspects in the next sections.

\section{Fusion and conjugation}\label{sec:fusion}	
In this section, we will describe operations of fusion and conjugation for $G\times G$-spaces with 
$D$-valued moment maps. By Theorem \ref{th:liftingtheorem}, these correspond to similar operations for $D/G$-valued moment maps.

\subsection{Fusion}
Given two Hamiltonian Poisson $G$-manifolds $(M_i,\pi_i)$ with moment maps $\Phi_i\colon M_i\to \g^*$, the direct product    
with diagonal $G$-action is again a Hamiltonian Poisson $G$-manifold, with moment map the pointwise sum $\Phi_1+\Phi_2\colon M_1\times M_2\to \g^*$. Given a Poisson Lie group structure on $G$, there is  a similar fusion operation for Hamiltonian $G$-manifolds with $G^*$-valued moment maps, using a pointwise product of the moment maps  \cite{lu:mo}  
-- however, the action on $M_1\times M_2$ is no longer the diagonal action \cite[Lemma 2.19]{fl:pc}. Similarly, given a metric on $\g=\on{Lie}(G)$, there is a fusion operation for quasi-Poisson $G$-spaces with $G$-valued moment maps \cite[Section 5]{al:qu}-- however, the bivector field on the product $M_1\times M_2$ must be modified. In \cite[Section 4]{sev:lef}, \v{S}evera describes a fusion operation for spaces with $D/G$-valued moment maps -- however, the fusion product $M_1\fus M_2$ is not the direct product of spaces, in general.

We shall now describe a product operation for quasi-Poisson $G\times G$-spaces with $D$-valued moment map. By  Theorem \ref{th:liftingtheorem}, it translates into \v{S}evera's fusion operation for spaces with $D/G$-valued moment maps. 

\begin{theorem}\label{th:fusion}
For quasi-Poisson  $G\times G$-spaces with $D^{}$-valued moment map 
$\Phi_i\colon M_i\to D$, $i=1,2$,
the space 
\[
{M}_1\fus {M}_2=({M}_1\times {M}_2)\Big/(\{e\}\times G_\Delta\times \{e\})
\]
is naturally a  quasi-Poisson  $G\times G$-space with $D^{}$-valued moment map. The action on this space is 
induced by the action of $G\times \{e\}\times \{e\}\times G$ on ${M}_1\times {M}_2$, and the moment map 
${\Phi}_1\fus {\Phi}_2$ is induced by  $\Mult_D\circ ({\Phi}_1\times {\Phi}_2)$.  
\end{theorem}

\begin{proof}
The Courant algebroid $\AA_D$ has the structure of a $\ca{CA}$-groupoid \cite{al:pur,lib:dir}
\[ \AA_D=D\times 	 (\ol{\dd}\oplus\dd) \rra \dd.\]
As a groupoid, it is the direct product of the 
pair groupoid  $\ol{\dd}\oplus\dd\rra\dd$ with the group $D$; 
the graph of 
the groupoid multiplication defines a 
Courant morphism
\begin{equation}\label{eq:multiplication}
 T\colon \AA_D^{}\times \AA_D^{}\da \AA_D^{}\end{equation}
 with base map $\Mult_D$.  
Explicitly, 
\[ \left((d_1,(\zeta_1,\zeta_1')),(d_2,(\zeta_2,\zeta_2'))\right)\sim_T (d,\zeta,\zeta')\ \ 
\Leftrightarrow\ \ 
d=d_1d_2,\ \zeta_1=\zeta,\zeta_1'=\zeta_2,\zeta_2'=\zeta'.\]  
The kernel of this Courant morphism is the subbundle 
\[ \ker(T)=D^2\times (0\oplus \dd_\Delta\oplus  0).
\]
The Dirac structure $L_D=D\times(\g\oplus\g)$ is a subgroupoid of $\AA_D$. Nevertheless, $T$ is only 
`weakly Dirac' for the given Dirac structures because  
\begin{equation}\label{eq:kernelintersection}
\ker(T)\cap (L_D\times L_D)=D^2\times (0\oplus \g_\Delta\oplus 0)\end{equation}
is nontrivial. This will be remedied by taking the reduction with respect to  the action of $\{e\}\times G_\Delta\times\{e\}$ on $\AA_D^{}\times \AA_D^{}$. This action 
has isotropic generators, given exactly by the subbundle \eqref{eq:kernelintersection}. To simplify notation, we write 
$G_\Delta$ in place of $\{e\}\times G_\Delta\times\{e\}$, and denote the   Courant reduction by
by $\big(\AA_D^{}\times \AA_D^{}\big)\qu G_\Delta$. 
The Dirac structure $L_D\times L_D$ is $G_\Delta$-invariant, and contains 
the generators \eqref{eq:kernelintersection}. Using Proposition \ref{prop:part2} about the reduction of Dirac morphisms, 
it follows that $T$ descends to a Dirac morphism
\[ \begin{tikzcd} 
\left(\big(\AA_D^{}\times \AA_D^{}\big)\qu G_\Delta,\ (L_D\times L_D)\qu G_\Delta\right)
\arrow[r, dashrightarrow, "{T\qu G_\Delta}"]\arrow[d] & [2em] (\AA_D,L_D)\arrow[d]\\(D\times D)/G_\Delta \arrow[r]& D
\end{tikzcd}
\]

On the other hand, $(\T({M}_1\fus {M}_2),T(M_1\fus M_2)$ is the reduction of the Dirac structure 
$(\T(M_1\times M_2),T(M_1\times M_2))$ under the action of $G_\Delta$. 
Letting  
\[ {R}_i\colon (\T{M}_i,T{M}_i)\da \left( \AA_D^{},\   L_D\right)\]
denote the Dirac morphisms for the quasi-Poisson $G\times G$-spaces ${M}_i$, the reduction 
of the product ${R}= {R}_1\times {R}_2$ 
by $G_\Delta$ is a Dirac morphism
\[ \begin{tikzcd} 
\big(\T(M_1\fus M_2),\ T(M_1\fus M_2)\big)
\arrow[r, dashrightarrow, "{R\qu G_\Delta}"]\arrow[d] & [2em] \left(\big(\AA_D^{}\times \AA_D^{}\big)\qu G_\Delta,\ (L_D\times L_D)\qu G_\Delta\right)\arrow[d]\\{M}_1\fus {M}_2 \arrow[r]& (D\times D)/G_\Delta
\end{tikzcd}
\]
By composition, we obtain the desired Dirac morphism 
\[  {R}_1\fus {R}_2\colon \big(\T(M_1\fus M_2),\ T(M_1\fus M_2)\big) \da \left( \AA_D^{},\   L_D\right).\qedhere\]
\end{proof}
 
 It is clear from the construction that the fusion product is associative, in the sense that 
 \[ (M_1\fus M_2)\fus M_3\cong M_1\fus (M_2\fus M_3)\] 
 canonically. On the other hand, there is no $G\times G$-equivariant isomorphism between $M_1\fus M_2$ and $M_2\fus M_1$, in general. (See \cite[Section 4.2.2]{taw:th} for a concrete counter-example.)

 \begin{remark}
 By combining Theorem \ref{th:fusion} with the Lifting Theorem \ref{th:liftingtheorem}, one obtains a fusion operation for 
 quasi-Poisson $G$-spaces with $D/G$-valued moment maps $\Phi_i\colon M_i\to D/G$, by lifting to space $\wh{M}_i$ with $D$-valued moment maps, and then taking the quotient of $\wh{M}_1\fus \wh{M}_2$. The resulting space $M_1\fus M_2$ is the fiber product of $M_1$ with $D\times_G M_2$ over $D/G$.   	
 \end{remark}

 The fusion operation generalizes in various directions, with straightforward generalizations of the proof:
 \begin{itemize}
 	\item \emph{Internal fusion}: If $M$ is a 
 	quasi-Poisson $(G\times G)\times (G\times G)$-manifold  with $D\times D$-valued moment map, then  
 	\[{M}_{\on{fus}}={M}/(\{e\}\times G_\Delta\times \{e\})\] 
 	is a quasi-Poisson $G\times G$-manifold  with $D$-valued moment map ${\Phi}_{\on{fus}}$ induced by the pointwise product. 
 	\item Given another $G'$-equivariant moment map target $\Q$ for a Manin pair $(\dd',\g')$ 
 	one has internal fusion of quasi-Poisson 
 	$(G\times G)\times (G\times G)\times G'$-spaces with $D\times D\times \Q'$-valued moment map, producing a space with $ D\times \Q'$-valued moment map. 
 	\item Given Lagrangian subalgebras $\g_1,\g_2\subset\dd$ with corresponding groups $G_1,G_2$, we may 
 	define quasi-Poisson $G_1\times G_2$-spaces with $D$-valued moment map, by using the 
 	Dirac structure $D\times (\g_1\oplus\g_2)$. One then has a fusion of  quasi-Poisson $G_1\times G_2$-spaces and	quasi-Poisson $G_2\times G_3$-spaces with $D$-valued moment maps to quasi-Poisson $G_1\times G_3$-spaces with $D$-valued moment maps.  
 	\item Suppose $\cc\subset \dd$ is a coisotropic subalgebra, and suppose that $\k=\cc^\perp$ integrates to a closed subgroup $K\subset D$. Let $\dd'=\cc/\k$ as in Section \ref{sec:quotientlifts}, with quotient map $\pi\colon \cc\to \dd'$. The fiber product 
 	\[ \mf{l}_\cc=\{(\zeta_1,\zeta_1)\in \cc\times \cc \colon \pi(\zeta_1)=\pi(\zeta_2)\}\]
    is a subgroupoid of the pair groupoid  	$\ol\dd\oplus\dd$, and is a Lagrangian Lie subalgebra. We may thus consider $D$ as a moment map target for $(\ol{\dd}\oplus \dd,\mf{l}_\cc)$. The argument for Theorem \ref{th:fusion} generalizes to give a 
    fusion product for quasi-Poisson $\mf{l}_\cc$-spaces with $D$-valued moment maps; as a manifold the fusion product of two such spaces is a product $(M_1\times M_2)/(\{e\}\times K_\Delta\times \{e\})$. 
    
    Using reduction with respect to the coisotropic subalgebra 
    $\ol\dd\oplus \cc$, we obtain the Manin pair $(\ol\dd\oplus \dd',\cc_\Delta)$, 
    where $\cc$ is embedded diagonally by the map   $\zeta\mapsto (\zeta,\pi(\zeta))$.    The quotient $D/K$ is a moment map target for this Manin pair. The lifting theorem gives an equivalence with quasi-Poisson $\cc$-spaces with $D/K$-valued moment map. 
 \end{itemize}
 

\subsection{Conjugation}
The conjugate $M^*$ of a Poisson $G$-space $(M,\pi)$ with  moment map $\Phi\colon M\to \g^*$ is obtained by changing the sign of $\pi$ and of the moment map. For Poisson Lie group actions with $G^*$-valued moment maps, conjugation involves replacing $\Phi$ with $\on{Inv}_{G^*}\circ \Phi$; similarly for $G$-valued moment maps. 

A conjugation for quasi-Poisson spaces with $D/G$-valued moment maps was introduced by Bursztyn-Crainic \cite[Proposition 4.3]{bur:di1} and \v{S}evera \cite[Section 6]{sev:lef}, independently. 

Using the Lifting Theorem, it may be understood in terms of a natural conjugation for quasi-Poisson spaces with $G\times G$-valued moment map $\Phi\colon M\to D$.  Denote by ${M}^*$ the same 
space with the new action 
\[ (g_1,g_2)\cdot_{\rm new}{m}=(g_2,g_1)\cdot {m}\]
obtained by switching the two factors in $G\times G$. Then ${\Phi}^*=\on{inv}_D\circ\ {\Phi}\colon {M}^*\to D$
is equivariant for the new $G\times G$-action. 
\begin{theorem}\label{th:conjugation}
	The space ${M}^*$ is naturally a quasi-Poisson $G\times G$-space with $D$-valued moment map ${\Phi}^*$. 
\end{theorem}
\begin{proof}
Let ${R}\colon \T{M}\da \AA_D^\op$ denote the Dirac morphism for the space ${M}$. 
It remains a Dirac morphism, if we change the sign of the metric for both Courant algebroids.
We define ${R}^*$ as the following composition, 
\[ \T {M}\stackrel{\cong}{\lra} \ol{\T {M}} \stackrel{{R}}{\da} D\times \ol{\dd\oplus\ol{\dd}} \stackrel{\cong}{\lra}
\AA_D^\op.\] 
Here the first isomorphism is  the involution $x=v+\mu\mapsto v-\mu$ of $\T {M}$; this preserves the Courant 
bracket, anchor, and  the Dirac structure $TM$, but changes the sign of the metric. 
The last isomorphism is the involution of $\AA_D^\op$ 
given by groupoid inversion $(d,\zeta',\zeta)\mapsto (d^{-1},\zeta,\zeta')$; this 
preserves the Courant bracket and anchor, as well as the Dirac structure $L_D$, but once again changes the 
sign of the metric. Thus, ${R}^*$ is a Courant morphism, with base map ${\Phi}^*$, and equivariant for the 
$G\times G$-action on ${M}^*$. 
\end{proof}	
It is clear from the construction that $({M}^*)^*\cong {M}$; furthermore,
	\[ ({M}_1\fus{M}_2)^*\cong {M}_2^*\fus {M}_1^*.\]
\begin{remark}	
	The quasi-Poisson $G$-spaces $M/(\{e\}\times G)$ and $M^*/(\{e\}\times G) $	are usually non-isomorphic as $G$-manifolds. 	See \cite[Section 4.2.2]{taw:th} for a concrete example. \end{remark}

\subsection{The quasi-symplectic case}\label{subsec:quasisymplectic}
In Section \ref{sec:quasisymplectic}  we discussed the differential form approach to quasi-symplectic $G$-spaces with $\Q$-valued moment maps, for a given choice of splitting $\alpha$ of  $\AA_{\Q}=\Q\times \dd$. For quasi-symplectic $G\times G$-spaces with $D$-valued moment maps, there is a canonical $D\times D$-equivariant splitting given by 
$\alpha=\hh(-\theta^R,\theta^L)$, where  $\theta^L,\,\theta^R\in \Omega^1(D,\dd)$ are the  Maurer-Cartan forms on $D$, and with the  3-form 
$\eta=\f{1}{12} \l\theta^L,[\theta^L,\theta^L]\r$. 
The resulting isomorphism $\AA\stackrel{\cong}{\lra} \T_\eta D$ reads as \cite{al:pur}
\[ (d,\zeta_1,\zeta_2)\mapsto \left(\zeta_2^L-\zeta_1^R+\hh \left(\l\theta^L ,\,\zeta_2\r+\l\theta^R,\, \zeta_1\r\right)\right)\Big|_d.\]

Using this identification, one may describe the conditions for a Hamiltonian quasi-symplectic $G$-spaces 
with $D/G$-valued moment map in terms of  a $G\times G$-invariant 2-form ${\omega}\in \Omega^2({M})$  satisfying the following axioms:
\begin{enumerate}[label=(\roman*)]
	\item  $\d\omega=-\f{1}{12}{\Phi}^* \l\theta^L,[\theta^L,\theta^L]\r$,
	\item $\iota((\xi_1,\xi_2)_{{M}}){\omega}=-\hh{\Phi}^*
	(\l \theta^L,\xi_2\r+\l\theta^R,\xi_1\r),\ \ \ \xi_1,\xi_2\in\g$.
	\item  $\ker(T{\Phi})\cap \ker({\omega})=\{0\}$.
\end{enumerate}
The quasi-symplectic form on the conjugate space $M^*$ is simply the opposite form $-\omega$. 
Given two such spaces $(M_i,\omega_i,\Phi_i)$, the quasi-symplectic 2-form ${\omega}$ on 
${M}_1\fus {M}_2$ is given by 
\begin{equation}\label{eq:omegafus}
 \pi^*{\omega}={\omega}_1+{\omega}_2-\hh \l {\Phi}_1^*\theta^L,\,{\Phi}_2^*\theta^R\r\end{equation}
where $\pi\colon {M}_1\times {M}_2\to {M}_1\fus {M}_2$ is the quotient map for the $G_\Delta$-action. This follows from 
the description of the `multiplication morphism' $T$
for the Courant algebroid $\AA_D$ in terms of the splitting, see \cite{al:pur}.
One may verify directly that the right-hand side of \eqref{eq:omegafus} is $G_\Delta$-basic, and that the resulting 2-form 
${\omega}$ satisfies the axioms (i),(ii).  
A direct proof of the  minimal degeneracy condition (iii) seems less straightforward; here the Dirac geometric approach is more convenient. 
%


\section{Examples from moduli spaces}\label{sec:moduliexamples}
In this section, we show that many examples of quasi-symplectic $G$-spaces with $D/G$-valued moment maps arise as moduli spaces of flat bundles over surfaces. In most cases, these are described in terms of their lifts to spaces with 
$D$-valued moment maps. These lifted spaces may be seen as special instances of the very general constructions of Li-Bland and \v{S}evera  \cite{lib:sypo,sev:mod}. See also the recent paper by \'{A}lvarez \cite{alv:poi} for related constructions.

\subsection{\v{S}evera formalism} As before, we denote by $\theta^L,\theta^R$ the Maurer-Cartan forms, and by $\eta\in \Omega^3(D)$ the Cartan 3-form. The 2-form 
\[ \beta=\hh \l d_1^*\theta^L,d_2^*\theta^R\r\in \Omega^2(D\times D)\]
satisfies $\d\beta=d_1^*\eta+d_2^*\eta-(d_1d_2)^*\eta$ under 
group multiplication. Here elements of $D^n$ are denoted $(d_1,\ldots,d_n)$, and $d_i$ is also used to denote projection to the $i$-th factor; thus $d_1\cdots d_i\colon D^n\to D$ is the product map. It will be convenient to introduce $d_0=(d_1d_2)^{-1}$; hence the exterior differential of $\beta$ may be written 
\[ \d\beta=\sum_{i=0}^2 d_i^*\eta.\]
The contractions with the vector fields on $D^2$, for $\zeta\in\dd$, 
\[ \zeta_{(0)}=(-\zeta^R,0),\ \zeta_{(1)}=(\zeta^L,-\zeta^R),\ \zeta_{(2)}=(0,\zeta^L)\] 
are given by 
\[ \iota(\zeta_{(i)})\beta=\hh \l d_i^*\theta^L+d_{i+1}^*\theta^R,\zeta\r,\ \ \ i=0,1,2,\]
with the convention $d_3=d_0$. Note that these formulas become more symmetric if $D^2$ is regarded as the 
submanifold of $D^3$ given by $d_0d_1d_2$. In particular, $\beta$ is unchanged under cyclic permutations of indices: $\beta=\hh \l d_0^*\theta^L, d_1^*\theta^R\r=\hh \l d_2^*\theta^L, d_0^*\theta^R\r$.


We will use the following construction due to 
\v{S}evera
	\cite{sev:mod}.  For any manifold $M$, define a  
	product on $C^\infty(M,D)\times \Omega^2(M)$ by 
	\[ (\Phi_1,\omega_1)\bullet (\Phi_2,\omega_2)=
	\left(\Phi_1\Phi_2,\omega_1+\omega_2-(\Phi_1,\Phi_2)^*\beta\right).\] 
	One checks that $\bullet$ is a group structure, with inverse $(\Phi,\omega)^{-1}=(\Phi^{-1},-\omega)$.  
	Furthermore, $\{e\}\times \Omega^2(M)$ is a central subgroup, and 
	$C^\infty(M,D)\times \Omega^2(M)\to \Omega^3(M),\ (\Phi,\omega)\mapsto \d\omega-\Phi^*\eta$ is a group homomorphism for this group structure.

\subsection{Representation spaces for surfaces with boundary}\label{sec:moduli1}
Let $\Sigma$ be a  compact, connected, oriented surface with non-empty boundary, with a finite set of marked points
$\ca{\V}\subset \p\Sigma$,  referred to as \emph{vertices}. We require at least one vertex on each boundary component.  
The vertices subdivide each boundary component into a finite union of oriented edges; let $\ca{E}$ be the set of such edges. 

Denote by $\Pi(\Sigma)\rra \ca{V}$ the fundamental groupoid, consisting of  homotopy classes of paths $\gamma\colon [0,1]\to \Sigma$ with endpoints in $\V$. The source and target maps are 
$\sz([\gamma])=\gamma(0),\ \tz([\gamma])=\gamma(1)$
and the groupoid multiplication is given by concatenation of paths.
Similarly, let $\Pi(\p\Sigma)\rra \ca{V}$ be the fundamental groupoid of the boundary. 
The inclusion of the boundary defines a groupoid morphism
$\Pi(\p\Sigma)\to \Pi(\Sigma)$. Given a Lie group 
$D$, let 
\[ M_D(\Sigma)=\Hom(\Pi(\Sigma),D)\] be the set of groupoid homomorphisms $\kappa\colon \Pi(\Sigma)\to D,\ [\gamma]\mapsto \kappa_\gamma$. This space is a smooth manifold: a choice of 
generators of $\Pi(\Sigma)$ gives a diffeomorphism to a product of several copies of $D$. 
Similarly, we may define a space $M_D(\p\Sigma)=\Hom(\Pi(\p\Sigma),D)$. Taking $\ca{E}\subset \Pi(\p\Sigma)$ as generators, 
this space is identified with  $D^{\ca{E}}=\on{Map}(\ca{E},D)$. 
Restriction $\Hom(\Pi(\Sigma),D)\to \Hom(\Pi(\p\Sigma),D)$ defines a smooth map 
\begin{equation}\label{eq:mommap} \Phi\colon M_D(\Sigma)\to  M_D(\p\Sigma)=D^{\ca{E}}.\end{equation}
This map is equivariant under the action of 
$D^{\ca{V}}=\on{Map}(\ca{V},D)$, where $d\colon \ca{V}\to D,\ {\mathsf{v}}\mapsto d_{\mathsf{v}}$ acts as 
on $\Hom(\Pi(\Sigma),D)$ as 
\begin{equation}\label{eq:action}
 (d\cdot \kappa)_\gamma=d_{\tz(\gamma)}\ \kappa_\gamma\ d_{\sz(\gamma)}^{-1},\end{equation}
 and similarly on $\Hom(\Pi(\p\Sigma),D)$. 
%
Let us now assume that the Lie algebra $\dd$ of $D$ comes equipped with an invariant metric $\l\cdot,\cdot\r$. 
Denote by $\theta^L,\ \theta^R$ the Maurer-Cartan forms and by $\eta\in \Omega^3(D)$ the Cartan 3-form. 
Let $\dd^{\ca{V}}=\on{Map}(\ca{V},\dd)$ denote the Lie algebra of $D^{\ca{V}}$; for $\zeta\in \dd^{\ca{V}}$ denote by 
$\zeta_\vz\in\dd$ its components.
\begin{theorem} \label{th:basic} 
The space $M_D(\Sigma)$ is a quasi-symplectic $D^{\ca{V}}$-manifold with $D^\E$-valued moment map. That is, 
it carries a $D^{\ca{V}}$-invariant 2-form $\omega$ with the following properties:
\begin{enumerate}
\item $\d\omega=-\sum_{\mathsf{e}\in \ca{E}} \Phi_{\mathsf{e}}^*\eta$. 
\item For $\zeta\in \dd^{\ca{V}}$, 
\[ \iota(\zeta_{ M_D(\Sigma)})\omega=-\hh\sum_{\ez\in \ca{E}}
\left(\l\Phi_{\ez}^*\theta^L,\ \zeta_{\sz(e)}\r+\l\Phi_{\ez}^*\theta^R,\, \zeta_{\tz(e)}\r\right).
\] 
 \item\label{it:c} $\ker(\omega)\cap \ker(T\Phi)=\{0\}$.   
\end{enumerate}
\end{theorem}

In terms of a formulation with bivector fields, the moduli spaces $M_D(\Sigma)$ are studied in the work of 
Li-Bland and \v{S}evera \cite{lib:sypo}. One way of constructing the 2-form starts from the Atiyah-Bott symplectic 2-form $\omega_{AB}$ on a moduli space of flat $D$-bundles over $\Sigma$, with trivializations at the vertices, and up to gauge transformations that are trivial along the edges. 
The method of \cite{al:mom,cab:dir} shows how to descend  $\omega_{AB}$ to a 2-form (no longer closed) on the space of holonomies. In particular \cite{cab:dir} explains the nondegeneracy condition \eqref{it:c} in terms of reduction of Dirac morphisms.

Consider first the case that 
$\Sigma$ is a \emph{disk}. Choose an initial vertex $v\in \ca{V}$, and label the edges clockwise as ${\mathsf{e}}_1,\ldots,{\mathsf{e}}_r$ (so, $\sz({\mathsf{e}}_{i+1})=\tz({\mathsf{e}}_i)$). 
	\begin{center}
	\includegraphics[scale=0.4]{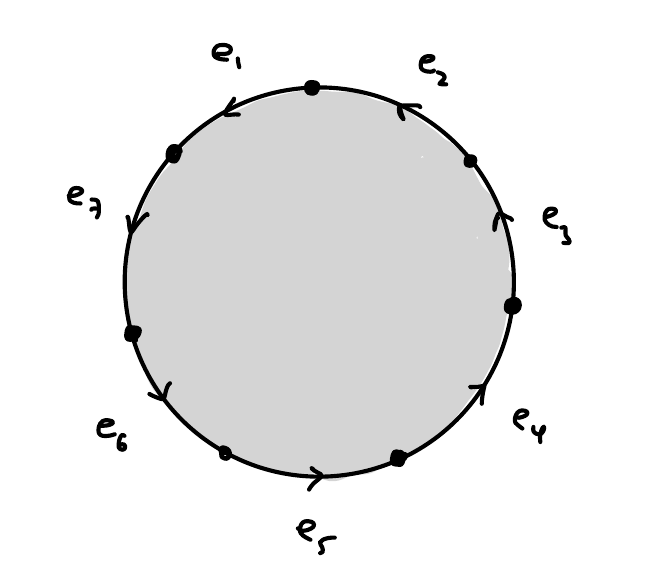}
\end{center}
Then $M_D(\Sigma)\subset D^r$ is the submanifold given by $\prod d_i=e$, 
with moment map $\Phi=(\Phi_1,\ldots,\Phi_r)$ the inclusion. Using \v{S}evera's formalism, 
the 2-form $\omega\in \Omega^2(M_D(\Sigma))$ is  given by 
\begin{equation} \label{eq:severa}
(e,\omega)=(\Phi_1,0)\bullet\cdots \bullet (\Phi_r,0).
\end{equation}
%
The case of a general surface $\Sigma$ can be obtained from that of a disk by making boundary identifications. 
To this end, cut the surface along paths, with end points in $\V$, until it becomes a disk (equivalently,  polyhedron) 
$\Sigma^c$. The quotient map $\Sigma^c\to \Sigma$
restricts to a surjective map $\V\to \V^c$ of vertices, and defines a surjective groupoid morphism $\Pi(\Sigma^c)\to \Pi(\Sigma)$. This defines an embedding  
\begin{equation}\label{eq:embedding}
j\colon M_D(\Sigma)\to M_D(\Sigma^c).\end{equation}
The 2-form $\omega\in M_D(\Sigma)$ is simply the pullback of the 2-form on $ M_D(\Sigma^c)$ under this embedding. 
In other words, it is given by the same formula \eqref{eq:severa}, by letting 
$\Phi_i=\Phi_j^{-1}$ for any two paired edges $\ez_i,\ez_j$ -- the moment map components $\Phi_\ez$  are the components $\Phi_i$ for non-paired edges. The 2-form $\omega$ does not depend on the choice of cutting used in the construction of $\Sigma^c$. 
One possible proof of this fact uses \eqref{eq:severa} to check that `elementary moves' (cutting the polygon along diagonals separating paired edges $\ez_1,\ez_2$, and re-gluing along these paired edges) do not change the 2-form. 

Alternatively, the 2-form may be described using a Dirac-geometric approach -- see Appendix \ref{app:D}. In particular, the non-degeneracy property \eqref{it:c} is automatic from that approach. In a nutshell, when $\Sigma$ is a disk, one may regard $M_D(\Sigma)$ as an orbit 
for a Dirac structure $E^{(r)}$ inside $(D\times (\ol{\dd}\oplus \ol{\dd}))^r$. 
In the general case one obtains $M_D(\Sigma)$ as a Dirac-geometric cross section.

\subsection{Surfaces with colored boundary}\label{sec:moduli2}
Suppose now that $\g\subset \dd$ is a Lagrangian Lie subalgebra of $\dd$, integrating to a closed Lie subgroup $G\subset D$. 
The pullback of the Cartan 3-form $\eta\in \Omega^3(D)$ to $G$ vanishes, since $\g$ is Lagrangian.   
Suppose $\Sigma$ (as in the previous section) has an even number of edges for each boundary component, alternatingly labeled as 
`free' or `colored':
\[ \E=\E_{\on{col}}\sqcup \E_{\on{free}}.\]  
A picture of such a surface is given in the introduction. 
Every vertex is either the source or the target of a uniquely defined free edge. That is, 
\begin{equation}\label{eq:st} 
\V=\tz( \E_{\on{free}})\sqcup \sz( \E_{\on{free}}).
\end{equation}
	
The submanifold $N\subset  D^\E$ 
of maps $\E\to D$ taking colored edges to $G$ is transverse to the $D^\V$-orbits, and hence 
is transverse to the moment map $\Phi\colon M_D(\Sigma)\to D^\E$. Since $N$ is $G^\V$-invariant, 
it follows that the pre-image
$ M^{\on{col}}_D(\Sigma)=\Phi^{-1}(N)$
is a $G^{\ca{V}}$-invariant submanifold 
\begin{equation}\label{eq:i}
 i\colon M^{\on{col}}_D(\Sigma)\to M_D(\Sigma).\end{equation}
Thus, $M^{\on{col}}_D(\Sigma)$ consists 
of all those homomorphims $\kappa\colon \Pi(\Sigma)\to D$ such that $\kappa(\ez)\in G$ for each colored edge $\ez\in\E_{\on{col}}$.  This space carries a $G^{\ca{V}}$-invariant 2-form $\omega^{\on{col}}$ given by pullback of $\omega$, and 
a map 
\begin{equation}\label{eq:coloredmoment} \Phi^{\on{col}}\colon M_D^{\on{col}}(\Sigma)\to D^{\ca{E}_{\on{free}}}\end{equation}
taking $\kappa$ to its values on the \emph{free} edges. By \eqref{eq:st}, $ G^{\ca{V}}=(G\times G)^{\E_{\on{free}}}$.

\begin{theorem}
	The space $(M_D^{\on{col}}(\Sigma),\omega^{\on{col}})$ is a quasi-symplectic $(G\times G)^{\E_{\on{free}}}$-space with $D^{\E_{\on{free}}}$-valued moment map \eqref{eq:coloredmoment}. Equivalently, there is an exact $(G\times G)^{\E_{\on{free}}}$-equivariant morphism of Manin pairs 
\[  R^{\on{col}}\colon \left(\T M^{\on{col}}_D(\Sigma),TM^{\on{col}}_D(\Sigma)\right)\da \left(\AA^{\E_{\on{free}}} ,D^{\E_{\on{free}}}\times (\g\oplus \g)^{\E_{\on{free}}}\right).\]
\end{theorem}

Hence, the quotient by $(\{e\}\times G)^{\E_{\on{free}}}$ is a quasi-symplectic $G^{\E_{\on{free}}}$-space with $(D/G)^{\E_{\on{free}}}$-valued moment map. 
\begin{proof} 
Our proof will use the Dirac-theoretic formulation, similar to Theorem \ref{th:basic2}. 
We obtain $R^{\on{col}}$ as a composition of the Dirac morphisms 
\[ R^{\on{col}}=T\circ R\circ \T i\]
where $\T i$ is defined by the inclusion \eqref{eq:i}, $R$ is the Dirac morphism \eqref{eq:basic2} for the uncolored surface, and 
\begin{equation}\label{eq:T1}
 T\colon (\AA^\E,D^\E\times \dd^\V)\da \left(\AA^{\E_{\on{free}}} ,D^{\E_{\on{free}}}\times (\g\oplus \g)^{\E_{\on{free}}}\right)
 \end{equation}
is defined as the product over all $T^\ez$ for $\ez\in \E$, given as identity morphism of $\AA^\ez$
for each free edge and as the morphism
$T^\ez\colon \AA^\ez\da 0$ defined by 
$L_D\subset \AA$ for each colored edge. To see that that this is indeed a Dirac morphism, 
recall that $(\g\oplus \g)^{\E_{\on{free}}}\cong \g^\V$, since every vertex is either the source or the target of a 
unique free edge.  Using this observation, we see that every element of $\{\pi(d)\}\times (\g\oplus \g)^{\E_{\on{free}}}$ is 
$T$-related to a unique element of $\{d\}\times \dd^\V$. 
\end{proof}

\subsection{Examples}
We now discuss various examples of the spaces $(M_D^{\on{col}}(\Sigma),\omega^{\on{col}})$ introduced above. 

\begin{example}[2-gon]
Suppose first that $\Sigma$ is a disk, whose boundary is subdivided into two edges $\ez_1,\ez_2$, where $\ez_1$ is free and $\ez_2$ is colored. 
	\begin{center}
	\includegraphics[scale=0.5]{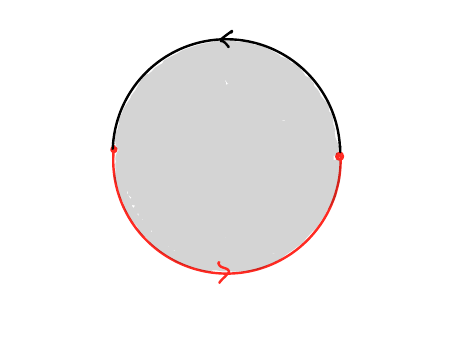}
\end{center}
We obtain a quasi-symplectic $G\times G$-space $M_D^{\on{col}}(\Sigma)$ consisting of 
pairs $(d,g)$ with $d\in D,\ g\in G$ and $dg=e$; the moment map takes such a pair to $d$. 
It follows that 
\[ M_D^{\on{col}}(\Sigma)= G,\ \ \omega^{\on{col}}=0\]
with $G\times G$-action on $d\in G$ given by $(a,a')\cdot d=ad(a')^{-1}$ and 
moment map the inclusion 
$G\hra D$. The corresponding $D/G$-valued space is $G/G=\pt$. 
\end{example}

\begin{example}[4-gon]
Consider next the case of a disk with four boundary edges, labeled clockwise as $\ez_1,\ldots,\ez_4$, where the even edges are colored.  

\begin{center}
	\includegraphics[scale=0.3]{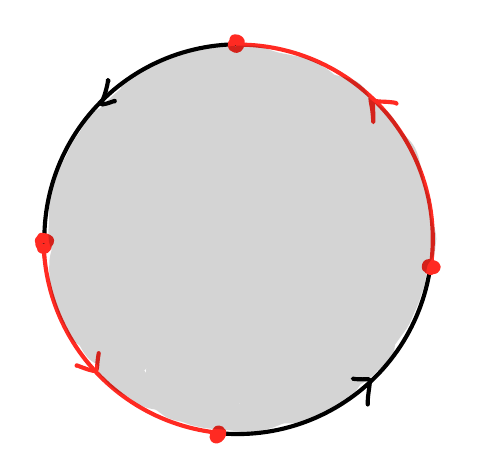}
\end{center}
The associated quasi-symplectic $(G\times G)^2$-space  $M_D^{\on{col}}(\Sigma)$ 
consists of 4-tuples $(d_1,g_1,d_2,g_2)$ with $d_1g_1d_2g_2=e$. Using $d=d_1,g_1,g_2$ as parameters, we obtain 
\[ M_D^{\on{col}}(\Sigma)\cong D\times G^2\]
with $(G\times G)^2$-action 
\[ (a_1,a_1',a_2,a_2')\cdot (d,g_1,g_2)=\left(a_1 d (a_1')^{-1},\,a_1'g_1a_2^{-1},\ a_2' g_2 a_1^{-1}\right),\]
moment map $(d,g_1,g_2)\mapsto (d,\ g_1^{-1}d^{-1} g_2^{-1})$, and 2-form 
\[ \omega=-\hh\left(\l d^*\theta^L,g_1^*\theta^R\r-\l d^*\theta^R,g_2^*\theta^L\r-
\l \Ad_d g_1^*\theta^R,g_2^*\theta^L\r
\right).\]
This space is the quasi-symplectic groupoid $(G\times G)\ltimes D\rra D$ discussed in Section \ref{sec:integration}.
\end{example}

\begin{example}[2n-gon] \label{ex:2ngon}
More generally, suppose $\Sigma$ is a disk whose boundary is subdivided into $2n$ edges, listed clockwise as 
$\ez_1,\ldots,\ez_{2n}$. Suppose the even edges are colored. 
 Then 
\[ M_D^{\on{col}}(\Sigma)=\{(d_1,g_1\ldots,d_n,g_{n})\in D^{2n}|\ d_{1}g_1\cdots g_nd_{n}=e,\ d_{i}\in D,\ g_i\in G\}\cong D^{n-1}\times G^n.\]
The moment map is given by projection to $(d_1,d_2,\ldots,d_{n})$.   The quotient under the action of $(\{e\}\times G)^n$ is identified with 
$D^n$, using the quotient map $(d_1,g_1\ldots,d_n,g_{n})\mapsto (d_1g_1,\ldots,d_ng_n)$. Hence, after renaming the variables, 
the resulting quasi-symplectic $G^n$-space with 
$(D/G)^n$-valued moment map is given by 
\[ M=\{(d_1,\ldots,d_n)\in D^n|\ d_1\cdots d_n=e\}\cong D^{n-1}\]
with the $G^n$-action 
\[ (a_1,\ldots,a_n).(d_1,\ldots,d_{n})=\left(a_1d_1a_2^{-1},\ldots,a_{n-1}d_{n-1}a_n^{-1},a_nd_na_1^{-1}\right)\]
and the moment map 
\[ M\to D/G\times \cdots \times D/G,\ \ (d_1,\ldots,d_n)\mapsto (d_1G,\ldots,d_nG).\]
These examples are similar to the \emph{mixed Poisson structures} of Lu-Mouquin \cite{lu:mix}. In particular, given a 
Lagrangian Lie subalgebra $\h$ complementary to $\g$, one obtains a Poisson structure on $D^n$ such that the given 
$G^n$-action is a Poisson Lie group action. 
\end{example}
\medskip 

One obtains more complicated examples by considering surfaces of higher genus or with more boundary components. 
\smallskip 
	
The fusion of moduli spaces for two colored surfaces is again a moduli space for a colored surface. The fusion operation may be depicted as follows: 
	\begin{center}
		\includegraphics[scale=0.3]{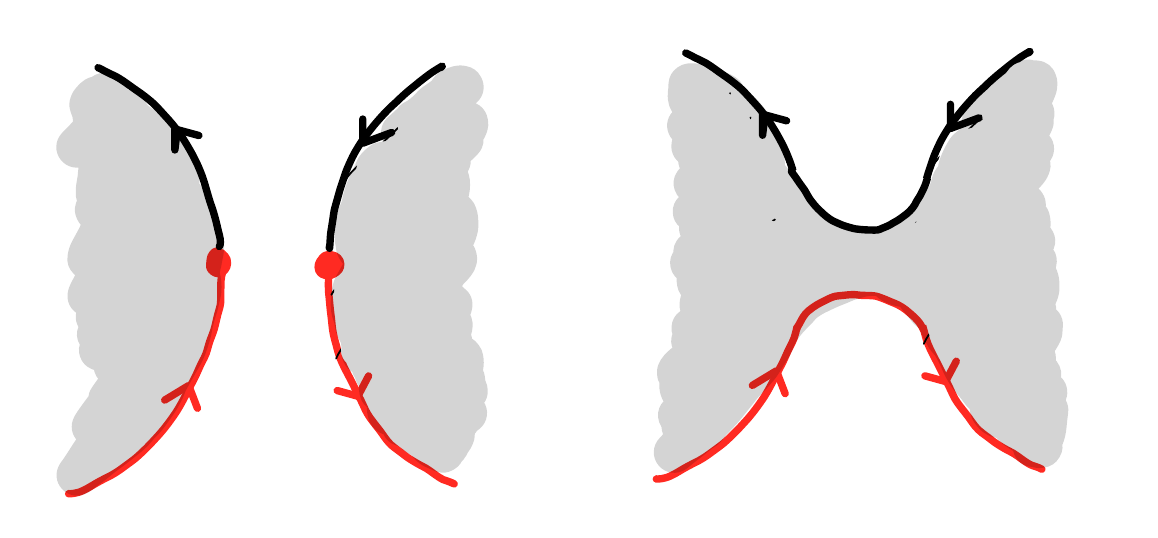}
	\end{center}
It may be viewed as the attachment of a bi-colored strip to the given pair of vertices, followed by removal of these vertices. Fusion of a $2k$-gon with a $2\ell$-gon gives a $2(k+\ell)-2$-gon. In particular,  the 2-gon acts as the identity under fusion, as illustrated below:
\begin{center}
\includegraphics[scale=0.38]{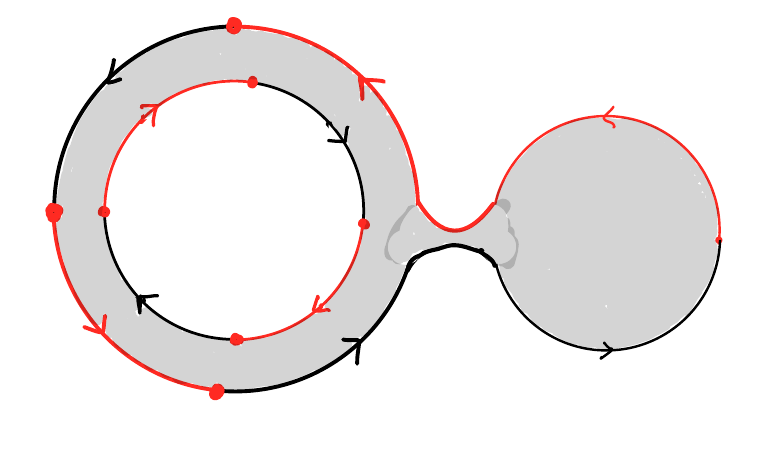}
\end{center}
One can obtain all (compact, connected, oriented) colored surfaces, except for the 2-gon, 
by fusions of 4-gons, provided that one also uses 
\emph{internal fusion} of two free edges of a given connected colored surface.  For example, the $2n$-gon moduli space (Example \ref{ex:2ngon}) is obtained by fusing $n-1$ copies of the 4-gon example. Internal fusion of the $2n$-gon moduli space 
produces 
a genus 0 surface with two boundary components, with $n_1$, $n_2$ free edges where $n_1+n_2=n-1$. 
Here, it is possible to have $n_1=0$ or $n_2=0$, corresponding to colored surfaces of a slightly more general kind than what was considered above: There may be colored boundary  components not containing vertices. See the following example.

\begin{example}
Consider again the case that $\Sigma$ is a 4-gon. Let $\wh{M}\cong D\times G^2$ be the 
corresponding quasi-symplectic $(G\times G)^2$-space $\wh{M}$ 
with $D^2$-valued moment map. 
By fusing the vertices $\tz(\ez_1)$ and $\sz(\ez_3)$,  
we	obtain a quasi-symplectic $G\times G$-space $\wh{M}_{\on{fus}}$ 
with $D$-valued moment map. This fused space  is a moduli space of the following colored surface: 
\begin{center}
	\includegraphics[scale=0.3]{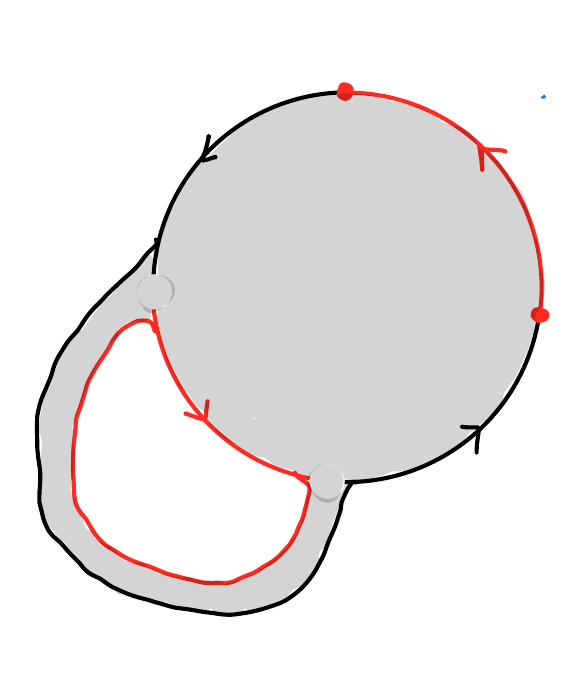}
\end{center}
More explicitly, $\wh{M}$ is the quasi-symplectic 
$(G\times G)^2$-space 
\[ \wh{M}=(d_1,g_1,d_2,g_2)|\ \d_1g_1d_2g_2=e\}\cong D\times G^2.\]
The fusion is described as the quotient under the action 
\[ h.(d_1,g_1,d_2,g_2)=(hd_1,g_1,d_2h^{-1},hg_2h^{-1}).\]
Hence, using $g_1,d_1,g_2$ as parameters on $\wh{M}$, we have 
$\wh{M}_{\on{fus}}\cong G\times (D\times G)/G$. The moment map and $G\times G$-action 
for the fused space are 
\[ \wh{\Phi}_{\on{fus}}\colon \wh{M}_{\on{fus}}\to D,\ \ \ 
(g_1,[d_1,g_2])\mapsto g_1^{-1}\Ad_{d_1^{-1}}g_2^{-1}\]
and  
\[ (a,a')\cdot  (g_1,[d_1,g_2])=(a' g_1 a^{-1},\,[d_1(a')^{-1},\ g_2]).\]
Taking the quotient by the $\{e\}\times G$-action, we obtain the space 
\[ M_{\on{fus}}=(D\times G)/G\]
(the bundle associated to the principal $G$-bundle $D\to D/G$ by the adjoint action of $G$ on itself -- that is, the bundle of gauge transformations), with moment map $[d,g]\mapsto (d^{-1}g^{-1}d)G$. 
\end{example}

Reduction of a space $M_D^{\on{col}}(\Sigma)$ with respect to one of the factors $D^\ez$, $\ez\in \E_{\on{free}}$ is given by 
\[ (\Phi^{\ez})^{-1}(G)/(G\times G),\]
where $\Phi^\ez$ is the component of the moment map corresponding to $\ez$. The reduced space is the space 
$M_D^{\on{col}}(\Sigma')$, where $\Sigma'$ is obtained from $\Sigma$ by removing the free edge $\ez$ along with its vertices: 
\begin{center}
	\includegraphics[scale=0.35]{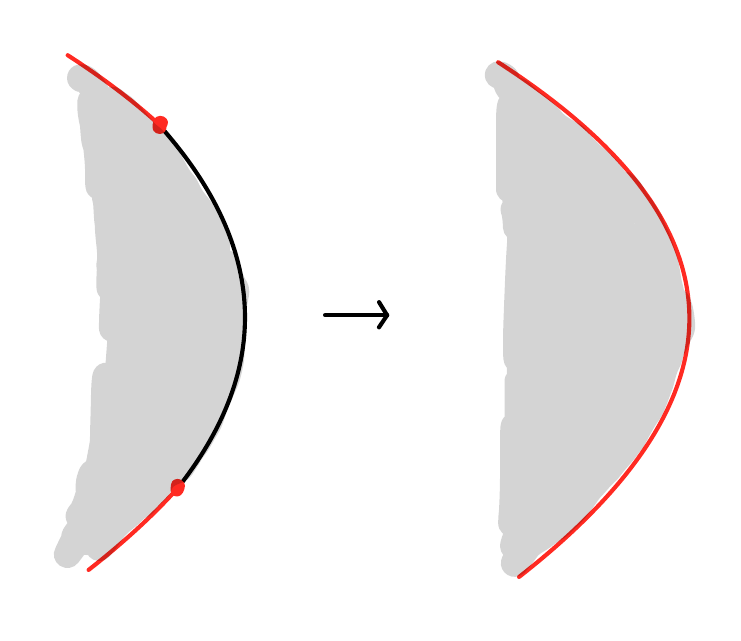}
\end{center}
Again, the resulting surface may have some colored boundary components with no vertices.

\begin{appendix}
	
\section{Background on Dirac geometry}\label{app:A}
\subsection{Dirac structures}
A \emph{Courant algebroid} 
\cite{cou:di,liu:ma}  is a vector bundle
$\AA\to M$ with a (non-degenerate)  fiber metric $\l\cdot,\cdot\r\colon \AA\times _M\AA\to \R$, an anchor map $\a\colon \AA\to TM$, and a bracket $\Cour{\cdot,\cdot}\colon \Gamma(\AA)\times \Gamma(\AA)\to \Gamma(\AA)$ on the space of sections. 
These are required to satisfy the following axioms, for all sections $\sigma_i\in \Gamma(\AA) $:
\begin{enumerate}
	\item $\Cour{\sigma_1,\Cour{\sigma_2,\sigma_3}}
	=\Cour{\Cour{\sigma_1,\sigma_2,\sigma_3}}+\Cour{\sigma_2,\Cour{\sigma_1,\sigma_3}}	$, 
	\item $\L_{\a(\sigma_1)}\l\sigma_2,\sigma_3\r=\l\Cour{\sigma_1,\sigma_2},\sigma_3\r+\l\sigma_2,\Cour{\sigma_1,\sigma_3}\r$,
	\item  $\Cour{\sigma_1,\sigma_2}+\Cour{\sigma_2,\sigma_1}=\a^*\d \l\sigma_1,\sigma_2\r$.
\end{enumerate}
Here the last line uses the metric to identify $\AA^*\cong \AA$. A \emph{Dirac structure} $(\AA,E)$ over $M$ is a Courant algebroid 
$\AA\to M$ together with a  Lagrangian subbundle $E\to M$ whose space of sections is closed under the Courant bracket. 
(If $\AA$ is given, one refers to $E$ as a Dirac structure.) 
A Courant algebroid over $M=\pt$ is just a metrized Lie algebra $\dd$, and a Dirac structure over $M=\pt$ is a Manin pair 
$(\dd,\g)$.  For this reason, Dirac structures $(\AA,E)$ are also called \emph{Manin pairs} \cite{liu:ma}. 

Let $\eta\in \Omega^3(M)$ be a closed 3-form on $M$.  
The \emph{standard ($\eta$-twisted) Courant algebroid}  is the direct sum 
$\T_\eta M=TM\oplus T^*M$, with metric given by the pairing, with anchor the projection to the first summand, and with 
Courant bracket on sections given by 
\begin{equation}\label{eq:courantbracket}
\Cour{X_1+\alpha_1,X_2+\alpha_2}=[X_1,X_2]+\L_{X_1}\alpha_2-\iota_{X_2}\d \alpha_1
+\iota_{X_1}\iota_{X_2}\eta.\end{equation}
For $\eta=0$ we also write $\T M$.
The graph $\on{gr}(\pi)$ of a Poisson bivector field is an example of a Dirac structure in $\T M$, as is the graph of any closed 2-form. 

 Another  type of examples are the \emph{action Courant algebroids} \cite{lib:cou}, for any action of a metrized Lie algebra $\dd$ on $M$ for which all stabilizers are coisotropic. As a vector bundle, the action Courant algebroid is $M\times \dd$, the metric and bracket extend the given ones on constant sections $\dd\subset \Gamma(M\times\dd)$, and the anchor is the action map $M\times \dd\to TM$. Given a Lagrangian Lie subalgebra $\g\subset \dd$, one obtains a Dirac structure 
 $(M\times \dd,\ M\times \g)$.

\subsection{Morphisms}

\begin{definition}\cite{al:der} (See also  \cite{al:pur,bur:cou,lib:dir}.) 
	\begin{enumerate}[label=(\roman*)]
		\item\label{it:i}
		Let $\AA_1,\AA_2$ be Courant algebroids over manifolds $M_1,M_2$.  A \emph{Courant morphism} 	
		\[ R\colon \AA_1\da \AA_2\]
		is a smooth map $\Phi\colon M_1\to M_2$ together with a Lagrangian subbundle $R\subset \AA_2\times \ol{\AA}_1$ 
		along $\on{gr}(\Phi)\subset M_2\times M_1$, with the property that the image of $R$ under the anchor is tangent to $\on{gr}(\Phi)$, and such that 
		the space of sections of $\AA_2\times \ol{\AA}_1$ whose restrictions to $\on{gr}(\Phi)$ takes values in $R$, is closed under 
		Courant bracket. Here $\ol{\AA}$ denotes the Courant algebroid obtained from $\AA$ by changing the sign of the metric.
		\item\label{it:ii} Let  $(\AA_i,E_i),\ i=1,2$ be Dirac structures over manifolds $M_i$. A \emph{Dirac morphism} 
		or \emph{morphism of Manin pairs} 
		\[ R\colon (\AA_1,E_1)\da (\AA_2,E_2)\]
		is a Courant morphism $R\colon \AA_1\da \AA_2$, with the property that $R\cap (E_2\times E_1)$ is the graph of a 
		bundle map $\Phi^*E_2\to E_1$. 
	\end{enumerate} 
\end{definition}
 The dotted arrow is used since $R$ is not a bundle map, but only a relation. We shall often write $x_1\sim_R x_2$ in place of $(x_2,x_1)\in R$.

\begin{remark}
As a consequence of $\a(R)\subset \on{gr}(T\Phi)$, one obtains 
\[ \a_1^*(\mu_1)\sim_R  \a_2^*(\mu_2)\]
for all $\mu_i\in TM_i^*|_{m_i}$ with $m_2=\Phi(m_1),\ (T\Phi|_{m_1})^*\mu_2=\mu_1$. 
\end{remark}

\begin{remark}
   	Condition 
	\ref{it:ii} for a Dirac morphism says that for a given point $(m_2,m_1)\in \on{gr}(\Phi)$, if 
	$x_2\in (E_2)_{m_2}$ then there is a \emph{unique} element  $x_1\in (E_1)_m$ with $x_1\sim_R x_2$. 
	We stress that we are working with the \emph{strong} version of Dirac morphisms, as opposed to the weak version where uniqueness is not required.  
\end{remark}

\begin{remark} 
	A vector bundle morphism $\Phi^*E_2\to E_1$ as in (ii) is the same as a vector bundle \emph{comorphism} $E_1\da E_2$. 
	For a Dirac morphism, the compatibility of $R$ with brackets implies that it is a \emph{ Lie algebroid comorphism}: In particular, 
	the pullback map on sections $\Gamma(E_2)\to \Gamma(E_1)$ preserves Lie brackets. 
\end{remark}

Suppose $M_i,\ i=1,2$ are manifolds with closed 3-forms 
$\eta_i$, defining standard Courant algebroid $\T_{\eta_i} M_i$. 
Given a map $\Phi\colon M_1\to M_2$ and a 2-form $\omega\in \Omega^2(M_1)$ satisfying
$\d\omega=\eta_1-\Phi^*\eta_2$, we define the \emph{standard Courant morphism} 
\[ \T_\omega\Phi\colon \T_{\eta_1} M_1\da \T_{\eta_2} M_2\] 
by 
\[ v_1+\mu_1\sim_{\T_\omega\Phi } v_2+\mu_2 \ \ \Leftrightarrow \ \ v_2=(T\Phi|_{m_1})(v_1),\  \mu_1=(T\Phi|_{m_1})^*\mu_2+\iota_{v_1}\omega.\] 
For $\omega=0$ we use the notation $\T\Phi$. 
Given Poisson structures 
$\pi_i$ on $M_i$, the map $\Phi$ is a Poisson map if and only if it defines a Dirac morphism $\T\Phi\colon (\T M_1,\on{gr}({\pi_1}))\da  (\T M_2,\on{gr}(\pi_2))$.

If $\O\subset M$ is an orbit of a Dirac structure $(\AA,E)$ (that is, a leaf of the singular foliation defined by $\a(\Gamma(E))\subset \Gamma(TM)$), there is a canonical Dirac morphism 
\begin{equation}\label{eq:orbit}
 R\colon (\T\O,T\O)\da (\AA,E).
 \end{equation}
 As a subbundle of $\AA\times \ol{\T \O}$ along the graph of the inclusion $\iota_\O$, it is spanned by the restrictions of $(\sigma,\a(\sigma))$ for $\sigma\in \Gamma(E)$, together with the restrictions of $(\a^*\mu,\iota_\O^*\mu)$ for $\mu\in \Omega^1(M)$. 

 As another example, if $(\dd,\g)$ is a Manin pair, 
an  equivariant map $\Phi\colon M_1\to M_2$ of $\dd$-manifolds 
with coisotropic stabilizers gives a Dirac morphism 
\begin{equation}
(M_1\times \dd,\ M_1\times \g)\da (M_2\times \dd,\ M_2\times \g).
\end{equation}
In particular, given a $\dd$-action on $M$ with coisotropic stabilizers, we have the Dirac morphism $(M\times \dd,\,M\times \g)\da 
(\dd,\g)$ defined by the map $M\to \pt$. Other examples arise from the theory of Dirac Lie groups and Dirac homogeneous spaces (cf.~ \cite{lib:dir,me:dir}).

\subsection{Exact Courant algebroids}\label{subsec:exact}
A Courant algebroid $\AA\to M$ is called  \emph{exact} \cite{sev:let} if the sequence $0\to T^*M\to \AA\to TM\to 0$ is exact. 
Equivalently, $\on{rank}(\AA)=2\dim M$ and the 
anchor $\a\colon \AA\to TM$ is onto. 
In this case,  the choice of a Lagrangian splitting $j\colon TM\to \AA$ of the anchor gives an identification with the standard Courant algebroid  
$\AA\cong \T_\eta M$
for a closed 3-form $\eta\in \Omega^3(M)$. Here $\eta$ is given in terms of $j$  by 
\[ \iota_{X_1}\iota_{X_2}\iota_{X_3}\eta=\l j(X_1),\Cour{j(X_2),j(X_3)}\r.\]  The space of all splittings is an affine space
over the space of 2-forms $\varpi\in \Omega^2(M)$;  any two splittings are related by 
\begin{equation}\label{eq:changeofsplitting}
j'(X)=j(X)-\a^*(\iota(X)\varpi);\end{equation}
the resulting 3-form changes to 
\begin{equation}\label{eq:changeofeta} \eta'=\eta+\d\varpi.\end{equation}
A morphism $R\colon \AA_1\da \AA_2$ of exact Courant algebroids,  
with base map $\Phi\colon M_1\to M_2$,  is called \emph{exact} \cite{cab:dir,lib:sypo}
if the sequence  
\[ 0\to \on{gr}(T^*\Phi)\to R\to \on{gr}(T\Phi)\to 0\] 
is exact. Equivalently, the anchor identifies $\a(R)= T\Phi$ as a relation
$\a(R)\colon TM_1\da TM_2$.  The choice of  splittings $\AA_1\cong \T_{\eta_1} M_1,\ \ \AA_2\cong \T_{\eta_2} M_2$ 
describes the exact Courant morphisms as $R=\T_\omega\Phi$. 

\begin{example}
An exact Courant morphism $\T_\eta M\da 0$ is the same as a 
2-form $\omega$ satisfying $\d\omega=\eta$; one obtains 
$\on{gr}(\omega)$ as the backward image of $0$. An exact Dirac morphism 
$(\T M,TM)\da (0,0)$ amounts to having a symplectic 2-form on $M$. 	
\end{example}

\begin{remark}
	Given a Dirac structure  $E\subset \AA$  in an exact Courant algebroid, with a source-simply connected  
	integration $\G\rra M$ of the  Lie algebroid 
	$E\Rightarrow M$, there is a canonically defined exact Dirac morphism  
	\[R\colon  (\T \G,T\G)\da (\AA,E)\times (\AA,E)^{\on{op}}\]
	with base map $(\tz,\sz)$. 
	See \cite{bur:int}, \cite[Section 4.2]{bur:cou} and \cite[Remark 6.2.1]{lib:th}. 
\end{remark}

\section{Reduction of Dirac morphisms}	\label{app:B}
We summarize and extend some facts about reductions of Courant algebroids, Dirac structures, and their morphisms.  
See \cite{alv:int,bur:red,bur:di,cab:dir,balibanu2023reduction} for related constructions and 
further details.

\subsection{Reduction of Dirac structures}
Suppose $\AA\to M$ is a Courant algebroid, on which a Lie group $G$ acts by Courant morphisms. We assume that the action 
on $M$ is a principal action, and that the action on $\AA$ has \emph{isotropic generators}
\[ \varrho\colon M\times \g\to \AA.\]
That is, $\varrho$ is $G$-equivariant (where the action on $\g$ is the adjoint action), and for all $\xi\in \g$ we have  
$\l\varrho(\xi),\varrho(\xi)\r=0$ and 
\[ \Cour{\varrho(\xi),\sigma}=\L_\xi \sigma\]
for all $\sigma\in \Gamma(\AA)$. 
Then the quotient 
\[ \AA_{\on{red}}=\f{\on{ran}(\varrho)^\perp}{\on{ran}(\varrho)}\Big/G\to M/G\]
is a Courant algebroid \cite{bur:red}. The quotient procedure comes with a Courant morphism $S\colon \AA\da \AA_{\on{red}}$, where 
$x\sim_S y$ if and only if $x\in \on{ran}(\varrho^\perp)$, with $y$ its image under the quotient map. 
If $E\subset \AA$ is a $G$-invariant Dirac structure satisfying 
$\on{ran}(\varrho)\subset E$, then 
\[ E_{\on{red}}=\f{E}{\on{ran}(\varrho)}\Big/G\]
is a Dirac structure in $\AA_{\on{red}}$.  (Note that $E\subset \on{ran}(\varrho)^\perp$.) 
This sets up a 1-1 correspondence between $G$-equivariant Dirac structures in $\AA$  containing the generators, and Dirac structures in $\AA_{\on{red}}$. 

\begin{remark}
Note that $S\colon \AA\da \AA_{\on{red}}$ is only a weak Dirac morphism for the Dirac structures $E$ and $E_{\on{red}}$, since 
for a given $y\in E_{\on{red}}|_{\pi(m)}$, the element $x\in E|_m$ with $x\sim_S y$  is not unique: it is only unique up to elements of $\on{ran}(\varrho)$. On the other hand, $E$ is the backward image of $E_{\on{red}}$: if $x\in \AA$ with $x\sim_S y\in E_{\on{red}}$ 
then $x\in E$. 
\end{remark}

\begin{remark}
If a second Lie group $H$ acts on $\AA$ by automorphisms, commuting with the  action of $G$ and preserving the generators for the $G$-action, then this action descends to the reduction. 
\end{remark}

\subsection{Reduction of Dirac morphisms}
Let $\AA_i\to M_i$ be $G$-equivariant exact Courant algebroids, 
 with isotropic generators $\varrho_i\colon M_i\times \g\to \AA_i$. 
We assume that the base actions are principal, hence the reductions 
$\AA_{i,\on{red}}\to M_i/G$ are defined. 

 \begin{proposition}\cite[Theorem 3.8]{cab:dir}  \label{prop:part1}
 There is a  1-1 correspondence between 
 	\begin{itemize}
 		\item Courant morphisms 
 		$ R_{\on{red}}\colon\AA_{1,\on{red}}\da \AA_{2,\on{red}}$
 		\item $G$-equivariant Courant morphisms 
 		$R\colon \AA_1\da \AA_2$ intertwining the generators: 
 		\[ \varrho_1(\xi)\sim_R \varrho_2(\xi),\ \ \ \xi\in\g.\]
 	\end{itemize}
 	Furthermore, $R_{\on{red}}$ is exact if and only if $R$ is exact.  
 \end{proposition}

In more detail, the correspondence is as follows. Let $\Phi\colon M_1\to M_2$ be the base map of $R$.
Let $S_i\colon \AA_i\da \AA_{i,\on{red}}$ be the reduction morphisms. 
In one direction, given $R$, we have $y_1\sim_{R_{\on{red}}} y_2$ if and only if 
there are elements $x_i\in\AA_i$ such that 
\[ x_1\sim_R x_2,\ \ x_i\sim_{S_i} y_i.\]
In the opposite direction,  given $R_{red}$, we have $x_1\sim_R x_2$ if and only if there are elements $y_i\in \AA_{i,\on{red}}$ and $\mu_i\in \ T^*M_i$ such that the following conditions hold:
\begin{itemize}	
	\item 	$ y_1\sim_{R_{\on{red}}} y_2$,
		\item  $x_i-\a_i^*(\mu_i)\sim_{S_i}y_i$. 
	\item $(T\Phi)(\a_1(x_1))=\a_2(x_2)$, 
    \item   $\mu_1=(T\Phi)^*(\mu_2)$,
\end{itemize}

The $\Rightarrow$ direction of the following result is also proved in \cite{cab:dir}. The other direction is not explicitly stated there; we will therefore provide a proof.

\begin{proposition}\label{prop:part2}
Let $R\colon \AA_1\da \AA_2$ be as in Proposition \ref{prop:part1}. Let $E_i\subset \AA_i$ be $G$-invariant Dirac structures 
containing the generators $\on{ran}(\varrho_i)$, with reductions $E_{i,\on{red}}$. Then 
 $R$ defines a  Dirac morphism
\[ R\colon (\AA_1,E_1)\da (\AA_2,E_2)\] 
if and only if  $R_{\on{red}}$ defines a Dirac morphism 
\[ R_{\on{red}}\colon (\AA_{1,\on{red}},E_{1,\on{red}})\da (\AA_{2,\on{red}},E_{2,\on{red}}).\]
\end{proposition}
\begin{proof}
Let $\pi_i\colon M_i\to M_i/G$ be the quotient maps. 	Suppose $R$ is a Dirac morphism. To see that $R_{\on{red}}$ is a Dirac morphism, let 
\[ y_2\in E_{2,\on{red}}\] be given, with base point $\pi_2(m_2)$, where $m_2=\Phi(m_1)$. 
Choose $x_2\in E_2|_{m_2}$ with $x_2\sim_{S_2} y_2$. Since $R$ is a Dirac morphism, there is a unique $x_1\in E|_{m_1}$ with $x_1\sim_R x_2$. Its image $y_1\in E_1$ satisfies $y_1\sim_{R_{\on{red}}} y_2$. 
To verify the uniqueness condition of Dirac morphisms, let $y_1\in E_{1,\on{red}}$ with 
$y_1\sim_{R_{\on{red}}} 0$. By definition of $R_{\on{red}}$, there exist elements $x_i\in  \AA_i$ with 
$x_1\sim_{S_1}y_1,\ x_2\sim_{S_2}0$, and such that 
 $x_1\sim_R x_2$. It is automatic that $x_i\in E_i$.  Since $x_2\sim_{S_2}0$, there is $\xi\in\g$ with $x_2=\varrho(\xi)|_{m_2}$. Subtracting $\varrho_i(\xi)$ from both $x_1,x_2$, we may arrange $x_2=0$. 
 But since $R$ is a Dirac morphism,  $x_1\sim_R 0$ shows that $x_1=0$, hence $y_1=0$. 

Conversely, suppose $R_{\on{red}}$ is a Dirac morphism. To show that $R$ is a Dirac morphism, let 
\[ x_2\in E_2|_{m_2}\]
be given,  where $m_2=\Phi(m_1)$. 
Let $y_2\in E_{2,\on{red}}|_{\pi_2(m_2)}$ be the image of $x_2$, and let $y_1\in E_{1,\on{red}}|_{\pi_1(m_1)}$ be the unique element such that 
$y_1\sim_{R_{\on{red}}} y_2$. By definition of $R_{\on{red}}$, there are elements $x_1'\in \AA_1|_{m_1},\ x_2'\in \AA_2|_{m_2}$ with 
with
\[ x_i'\sim_{S_i}y_i,\ \ \ x_1'\sim_R x_2'.\]
The first condition implies $x_i'\in E_i$. Since both $x_2,x_2'$ are $S_2$-related to $y_2$, there is a 
(unique) element $\xi$ such that $x_2'-x_2=\varrho_2(\xi)|_{m_2}$. Then $x_1=x_1'-\varrho_1(\xi)|_{m_1}\in E_1|_{m_1}$ satisfies
$ x_1\sim_Rx_2$. To verify the uniqueness condition, let $x_1\in E_1|_{m_1}$ with $x_1\sim_R 0$. Then its image $y_1\in E_{1,\on{red}}$ satisfies $y_1\sim_{R_{\on{red}}} 0$. Since $R_{\on{red}}$ is a Dirac morphism, it follows that $y_1=0$. This means $x_1=\varrho_1(\xi)|_{m_1}$ for some $\xi\in \g$. But then 
 $\varrho_1(\xi)\sim_R 0$ implies that 
 \[ T\Phi|_{m_1} (\xi_{M_1}|_{m_1})=0.\]
 Since the action is principal and $\Phi$ is $G$-equivariant, it follows that $\xi=0$, and so $x_1=0$.  
\end{proof}

\begin{remark}
	Given actions of another Lie group $H$ on $\AA_i$, commuting with the $G$-actions and preserving the Dirac structures and the generators, and if $R$ is $H\times G$-equivariant, then the reduced morphism $R_{\on{red}}$ is $H$-equivariant.  
\end{remark}

\begin{remark}[Bivector fields]
	\label{rem:bivacA}
In the setting of Proposition \ref{prop:part2}, suppose $F_2\subset \AA_2$ is a Lagrangian complement to $E_2$, 
with the property that $F_2\cap \on{ran}(\varrho_2)^\perp$ is $G$-invariant. Then its backward image $F_1$ under $R$ is 
a Lagrangian complement to $E_1$, and  $F_1\cap \on{ran}(\varrho_1)^\perp$ is $G$-invariant. The subbundles $F_i$ descend to Lagrangian complements $F_{i,\on{red}}$ to $E_{i,\on{red}}$, and $F_{1,\on{red}}$ is the backward image of $F_{2,\on{red}}$ under $R_{\on{red}}$. This means that $\Phi,\ \Phi_{\on{red}}$  
and the quotient maps $M_i\to M_i/G$  intertwine the bivector fields associated to the Lagrangian splittings.
\end{remark}

\subsection{Dirac Cross Section Theorem}\label{subsec:cross}
If the anchor $\a$ of the Courant algebroid $\AA\to M$ is transverse to a closed submanifold  $\iota\colon N\to M$, then the pull-back Courant algebroid 
\[ \iota^!\AA=\f{\a^{-1}(TN)}{\a^{-1}(TN)^\perp}\to N\] is defined (see e.g. \cite{lib:cou}). 
Similarly, given a Dirac structure 
$(\AA,E)$ such that $\a|_E$ is transverse to $N$, the pullback Dirac structure 
\[ \iota^!E=E\cap\a^{-1}(TN)\subset \iota^!\AA\]
is defined \cite{lib:cou}. (Observe that $E|_N+\a^{-1}(TN)=\AA|_N$ implies 
$ E\cap\, \a^{-1}(TN)^\perp=\{0\}$.)
We shall write $\iota^!(\AA,E)=(\iota^!\AA,\iota^!E)$.  As a special case, 
\begin{equation}\label{eq:specialcase} \iota^!(\T M,TM)=(\T N,TN).\end{equation}
The following result may be seen as a Dirac-geometric analogue of the Symplectic Cross Section Theorem of Guillemin-Sternberg \cite{gu:sy}. Several variations and special cases have appeared in the literature (e.g., \cite{al:qu,balibanu2023reduction,cro:log}).  

\begin{theorem}[Dirac Cross Section Theorem] \label{th:cross}
Let $R\colon (\AA_1,E_1)\da (\AA_2,E_2)$ be a Dirac morphism, with base map $\Phi\colon M_1\to M_2$. Suppose $ N_2\subset M_2$ is a closed submanifold transverse to $\a_2|_{E_2}$: 
\begin{equation}\label{eq:transversality} TM_2|_{N_2}=\a_2(E_2)|_{N_2}+TN_2.\end{equation}
Then 
$\Phi$ is transverse to $N_2$, and the pre-image $N_1=\Phi^{-1}(N_2)\subset M_1$  is transverse to $\a_1|_{E_1}$. 
Furthermore, letting $\iota_i\colon N_i\to M_i$ be the embeddings,  
$R$ induces a Dirac morphism 
\begin{equation}
	 \iota^!R\colon \iota_1^!(\AA_1,E_1)\da \iota_2^!(\AA_2,E_2),\end{equation}
with base map $\Phi|_{N_1}\colon N_1\to N_2$. 
\end{theorem}

\begin{proof}	
Consider $m_2=\Phi(m_1)$. Since $R$ is a Dirac morphism, every $x_2\in E_2|_{m_2}$ determines a unique $x_1\in E_1|_{m_1}$ with $x_1\sim_R x_2$. Applying the anchor, it follows that 
$\a_2(x_2)=(T\Phi|_{m_1})(\a(x_1))$. This shows  $\a_2(E_2|_{m_2})\subset \on{ran}(T\Phi|_{m_1})$, hence transversality of 
$\a_2|_{E_2}$ to $N_2$ implies transversality of $\Phi$ to $N_2$.  It follows that 
$N_1=\Phi^{-1}(N_2)$ is a submanifold.

The normal bundle functor gives an isomorphism $\nu(M_1,N_1)|_{m_1}\to \nu(M_2,N_2)|_{m_2}$.  
By assumption, every element of $\nu(M_2,N_2)|_{m_2}$ is realized as $\a(x_2)$ for some $x_2\in E_2|_{m_2}$; the unique
$x_1\in E_1|_{m_1}$ with $x_1\sim_R x_2$ realizes the corresponding element of $\nu(M_1,N_1)_{m_1}$. This shows 
that $N_1$ is transverse to $\a_1|_{E_1}$. In particular, the Dirac structures $\iota_i^!(\AA_i,E_i),\ i=1,2$ are defined. 
		
Write $\AA=\AA_2\times \ol{\AA_1}$, with anchor $\a=(\a_2,\a_1)$, and let 
$N=N_2\times N_1,\ M=M_2\times M_1$, with inclusion $\iota=(\iota_2,\iota_1)\colon N\to M$. 
Let $\Phi_N=\Phi|_{N_1}\colon N_1\to N_2$ be the restriction of $\Phi$.

We claim that 
$R|_{\on{gr}(\Phi_N)}\cap \a^{-1}(TN)$ restricts to a subbundle of $\AA$ along $\on{gr}(\Phi_N)$. This then implies that $R|_{\on{gr}(\Phi_N)}\cap \a^{-1}(TN)^\perp$ is a subbundle
as well, and hence so is the quotient 
\[ {\iota^! R}=\f{R|_{\on{gr}(\Phi_N)}\cap \a^{-1}(TN)}{R|_{\on{gr}(\Phi_N)}\cap \a^{-1}(TN)^\perp}\subset \iota^!\AA.\]		
To prove the claim, it is convenient to work with a decomposition of $R$. Since 
	$TM_2|_{N_2}=\a_2(E_2|_{N_2})+TN_2$ we have $\AA_2|_{N_2}=E_2|_{N_2}+\a_2^{-1}(TN_2)$. 
Choose a 	
Lagrangian subbundle $F_2\subset \AA_2$ with $\AA_2=E_2\oplus F_2$ and such that 
\[ F_2|_{N_2}\subset \a_2^{-1}(TN_2).\] The backward image 
\[ F_1=F_2\circ R\subset \AA_1\] is a smooth Lagrangian subbundle with $\AA_1=E_1\oplus F_1$ , 
and it comes with a bundle map 
$F_1\to F_2$ taking $y_1\in F_1$ to the \emph{unique} $y_2\in F_2$ such that $y_1\sim_R y_2$ (see e.g.~ \cite{al:pur}).
This determines a direct sum decomposition  
\[ R\cong F_1\oplus \Phi^*E_2,\]
where the two summands are embedded as graphs of the bundle maps $F_1\to F_2$ 
and $\Phi^*E_2\to E_1$. 
Since $(T\Phi)\big(\a_1(y_1)\big)=\a_2(y_2)$ for $y_i\in F_i,\ y_1\sim_R y_2$, 
and since $F_2|_{N_2}\subset \a_2^{-1}(TN_2)$, 
we have that 
\[ F_1|_{N_1}\subset \a_1^{-1}(TN_1).\] 
That is, $R|_{\on{gr}(\Phi_N)}\cap \a^{-1}(TN)$ contains the graph of the bundle map 
$F_1|_{N_1}\to F_2|_{N_2}$, and we obtain 
\begin{equation}\label{eq:RNdirsum}
	 R|_{\on{gr}(\Phi_N)}\cap \a^{-1}(TN)\cong F_1|_{N_1} \oplus 
(\Phi|_{N_1})^* (E_2\cap \a_2^{-1}(TN_2)).
\end{equation}
%
 Since $E_2\cap \a_2^{-1}(TN_2)$ has constant rank, so does $R|_{\on{gr}(\Phi_N)}\cap \a^{-1}(TN)$. This proves that ${\iota^! R}$ is a well-defined 
 Lagrangian subbundle along $\on{gr}(\Phi_N)$, and that its image under the anchor is tangent to 
 $\on{gr}(\Phi_N)$. It is involutive by reduction: the bracket of two sections of $\AA$ that restrict to sections of $R$ again restricts to a section of $R$, and similarly for sections that restrict to sections of $\a^{-1}(TN)$. 

It remains to show that ${\iota^! R}$ defines a Dirac morphism. Consider 
$m_2=\Phi(m_1)$ for $m_1\in N_1$, and let  $ [x_2]\in \iota_2^!E_2|_{m_2}$ be an element represented by 
$x_2\in E_2\cap \a_2^{-1}(TN_2)$. Let $x_1\in E_1|_{m_1}$ be the unique element 
with $x_1\sim_R x_2$. Then $x_1\in E_1\cap \a_1^{-1}(TN_1)$, and 
$[x_1]\sim_{\iota^! R} [x_2]$. 

For the uniqueness property, suppose we are given $w_1\in \iota_1^!E_1$ with $w_1\sim_{\iota^!R} 0$. 
By definition of $\iota^!R$, this lifts to a relation $z_1\sim_R z_2$, where 
$(z_2,z_1)\in R\cap \a^{-1}(TN)$ 
with $[z_1]=w_1$ and $[z_2]=0$. Using the direct sum decomposition \eqref{eq:RNdirsum} we can write $z_i=x_i+y_i$ with $x_i\in E_i,\ y_i\in F_i$ and $x_1\sim_R x_2,\ y_1\sim_R y_2$. From 
$[z_2]=0$, it follows that $z_2\in \a_2^{-1}(TN_2)^\perp\subset F_2|_{N_2}$, hence $x_2=0$.
But then $x_1\sim_R 0$ implies that $x_1=0$ as well, since $R$ is a Dirac morphism. Similarly, $[z_1]=w_1\in \iota_1^!E_1$ shows that 
$y_1=z_1\in E_1|_{m_1}+\a_1^{-1}(TN_1)^\perp$, hence $y_1\in \a_1^{-1}(TN_1)^\perp$ and finally $w_1=[y_1]=0$.    
\end{proof} 

\begin{remark}[Exact case] \label{rem:standard}
	Consider the special case that $R$ is an exact Courant morphism between standard Courant algebroids, 
\[ R=\T_\omega\Phi\colon \T_{\eta_1}M_1\da \T_{\eta_2}M_2.\]	 
Then  pullback Courant algebroids are again standard, and 
\[ \iota^! R=\T_{\iota_1^*\omega}\Phi_N\colon \T_{\iota_1^*\eta_1}N_1\da \T_{\iota_2^*\eta_2}N_2.\]
\end{remark}

\section{Proof of Theorem
\ref{th:groupoid}}\label{app:C}
\begin{proof}
	Choose a basis $e_a$ of $\g$, and denote by $\theta^{a,R}$ the components of the Maurer-Cartan form. We shall use the 
	simplified notation $\iota_a$ in place of $\iota_{(e_a)_{\Q}}$. With this notation, 
	\[ \chi(e_a,\zeta)=\iota_a\l\alpha,\zeta\r.\]
	
	\begin{enumerate}
		\item Moment map property: The $G\times G$-action on $\G=G\times M$ is given by $(h_1.h_2).(g,q)=(h_1 g h_2^{-1},h.q)$. 
		The generating vector fields for the first action are $(-\xi^R,0)$, for $\xi\in\g$. We have, 
		\begin{align*} \iota_{(-\xi^R,0)}\omega&=-\A_g^*\left(
		\l\alpha,\xi\r+\chi(\xi,\theta^R)
		\right)\\
		&=-\A_g^*\Big(
		\l\alpha,\xi\r-\sum_a\theta^{a,R} \iota_a\l\alpha,\xi\r)
		\Big)
		\\
		&=-\tz^*\l\alpha,\xi\r.\end{align*}
		Similarly, 
		the generating vector fields for the second action are $(\xi^L,\xi_{\Q})$. Since $\xi_{\Q}$ is $\Ad_g$-related to 
		$(\Ad_g\xi)_{\Q}$, we obtain 
		\begin{align*}
		\iota_{(\xi^L,\xi_{\Q})}\omega&=
		\A_g^*\left(\l\alpha,\Ad_g\xi\r-\iota_{(\Ad_g\xi)_{\Q}}\l\alpha,\theta^R\r+
		\chi(\Ad_g\xi,\theta^R)
		\right)+\iota_{\xi_\Q}\beta(g)
		\end{align*}	
		The second and third term cancel. Using the definition of $\beta(g)$ (Equation \eqref{eq:betaq}), 
		the first and last term combine to $\l\alpha,\xi\r=\sz^*\l\alpha,\xi\r$. 
		\item Minimal degeneracy:  Suppose  $(-\xi^R|_g,v)\in T_gG\times T_q\Q$ lies in $\ker(\omega) \cap 
		\ker(T\tz)\cap 
		\ker(T\sz)$. Since it is in  $\ker(T\sz)$ we have $v=0$, and since it is in $\ker(T\tz)$ we have 
		$\xi_{\Q}|_q=0$, i.e. 
		$\xi\in \g_q$.  The property 
		$\iota((-\xi^R,0))\omega=-\tz^*\l\alpha,\xi\r$ shows that such an element is in $\ker(\omega_{g,q})$ if and only if $\xi\in 
		\on{ran}(\alpha_q)^\perp=\on{ran}(\alpha_q)$. Since $\on{ran}(\alpha_q)\subset \dd$ is a complement to $\dd_q$, it follows that $\xi=0$. This shows $\ker(\omega) \cap 
		\ker(T\tz)\cap 
		\ker(T\sz)=0$. 
		\item Multiplicative property: We have to show 
		\[ \delta\omega\equiv \partial_0^*\omega-\partial_1^*\omega+\partial_2^*\omega=0,\] 
		where $\partial_i\colon G\times G\times \Q$ are given by 
		\[ \partial_0(h,g,q)=(g,q),\ \ \partial_1(h,g,q)=(hg,q),\ \ \partial_2(h,g,q)=(h,g.q).\]
		Using the equivariance properties of these maps, we have 	
		
		\begin{align*}
		\iota_{( -\xi^R,0,0)}\partial\omega&=-\partial_1^*\iota_{( -\xi^R,0)}\omega+\partial_2^*\iota_{( -\xi^R,0)}\omega\\
		&=-\partial_1^*\tz^*\l\alpha,\xi\r+\partial_2^*\tz^*\l\alpha,\xi\r\\
		&=0,
		\end{align*}
		since $\tz\circ \partial_1=\tz\circ \partial_2$. Similarly, 
		\begin{align*}
		\iota_{(0,\xi^L,\xi_{\Q})}\partial\omega&=\partial_0^* \iota_{(\xi^L,\xi_{\Q})}\omega-\partial_1^*
		\iota_{(\xi^L,\xi_{\Q})}\omega\\
		&=\partial_0^*\sz^*\l\alpha,\xi\r-\partial_1^*\sz^*\l\alpha,\xi\r\\
		&=0
		\end{align*}
		since $\sz\circ \partial_0=\sz\circ \partial_1$. These calculation show $\delta\omega\in \Omega^2(G\times G\times \Q)$ 
		is of tridegree $(0,0,2)$. 
		The component of tridegree $(0,0,2)$ vanishes by the cocycle property for $\beta(g)$:
		\[ (\partial\omega)^{(0,0,2)}=\beta(g)-\beta(hg)+\A_g^*\beta(h)=0.\]
		\item Exterior differential: We finally show $\d\omega=\sz^*\eta-\tz^*\eta$. The component of bidegree $(0,3)$ 
		of $\d\omega\in \Omega^3(G\times M)$ is $\d\beta(g)=-\A_g^*\eta+\eta$, which agrees with the component of bidegree $(0,3)$ of 
		$\sz^*\eta-\tz^*\eta$. It hence suffices to compare the contractions with $-\xi^R$. Since $-\xi^R$ is $\tz$-related to $\xi_{\Q}$ and $\sz$-related to $0$, this amounts to showing 
		\[ \iota(-\xi^R)\d\omega=-\tz^* \iota_{\xi_{\Q}}\eta,	\]
		or equivalently	
		\[
		\L(-\xi^R)\omega=-\tz^*\left(\d\l\alpha,\xi\r+\iota_{\xi_{\Q}}\eta\right).\]
		Restricting the identity $\delta\omega=0$ to the slice $\{h\}\times (G\times \Q)$, we obtain the transformation property of $\omega$ under the left multiplication $\ell_h(g,q)=(hg,q)$, namely 
		\[ \omega-\ell_h^*\omega+\tz^*\beta(h)=0.\]
		Replacing $h$ with $\exp(t\xi)$, and taking a derivative this shows
		\[ \L(-\xi^R)\omega=-\tz^*\f{d}{d t}\big|_{t=0}\beta(\exp(t\xi)).\]
		The right hand side is computed in Lemma \ref{lem:hard}, and gives the desired result. 
	\end{enumerate}
\end{proof}

\section{Dirac-geometric approach to moduli spaces}\label{app:D}
Consider the setting of Section \ref{sec:moduli1}: In particular, $\Sigma$ is a compact, connected, oriented surface with boundary, with a collection $\V$ of base points (\emph{vertices}) on the boundary and resulting decomposition of the boundary into segments called \emph{edges}. 
We will give an alternative formulation of Theorem \ref{th:basic} using Dirac geometry. 
It will be convenient to introduce the notation $\AA=D\times (\ol{\dd}\oplus\dd)$, and \[ \AA^\E=\on{Map}(\E,\AA)=\prod_{\ez\in\E}\AA^\ez\] where 
$\AA^\ez$ is the copy corresponding to the edge $\ez$. Consider the Manin pair $\big((\ol{\dd}\oplus\dd)^\E,\dd^\V\big)$, 
where $ \dd^\V$ is regarded  as the Lie subalgebra consisting of maps $(\zeta,\zeta')\colon \E\to \ol{\dd}\oplus\dd$ such that $\zeta_{\ez_1}=\zeta'_{\ez_2}
\mbox{ whenever } \sz(\ez_2)=\tz(\ez_1)$. The inclusion map takes 
$\xi\colon \V\to \dd$  to the function $\ez\mapsto (\xi_{\tz(\ez)},\xi_{\sz(\ez)})$. This Manin pair defines a $D^\V$-equivariant 
Dirac structure 
\begin{equation}\label{eq:basic3}
(\AA^\E ,D^\E\times \dd^\V).\end{equation}
Theorem \ref{th:basic} is equivalent to the following:  
\begin{theorem}\label{th:basic2}
	The space 	$M_D(\Sigma)=\Hom(\Pi(\Sigma),D)$ comes with a $D^\V$-equivariant exact Dirac morphism  
	\begin{equation}\label{eq:basic2}
	R\colon 	(\T M_D(\Sigma),TM_D(\Sigma))\da (\AA^\E ,D^\E\times \dd^\V)\end{equation}
with base map $\Phi$. In terms of the identification $\AA\cong \T_\eta D$ we have that $R=\T_\omega \Phi$, 
where $\omega\in \Omega^2(M_D(\Sigma))$ is the 2-form on the moduli space. 
\end{theorem}
\begin{proof}[Proof] 
	We may assume $\Sigma$ is connected. As in Subsection \ref{sec:moduli1}, we cut the surface along embedded curves until we arrive at a polyhedron (disk) $\Sigma^c$; conversely, $\Sigma$ is obtained from $\Sigma^c$ 
	 by making some boundary identifications. 
	Since $\Phi^c\colon M_D(\Sigma^c)\to D^{\E^c}$ is an embedding as the
	$D^{\V^c}$-orbit of $(e,\ldots,e)\in D^{\E^c}$,  we obtain a Dirac morphism 
		\[ R^c\colon 	(\T M_D(\Sigma^c),TM_D(\Sigma^c))\da (\AA^{\E^c} ,D^{\E^c}\times \dd^{\V^c})\]
	with base map $\Phi^c$.  (See Example \ref{ex:orbits}.) 
	Let $Q\subset D^{\E^c}$ be the submanifold defined by the equations $d_{\ez}=\d_{\ez'}^{-1}$ for any paired edges $\ez,\ez'\in \E^c-\E$. 
	Then 
		\[ M_D(\Sigma)=(\Phi^c)^{-1}(Q)\subset M_D(\Sigma^c)\]
	see \eqref{eq:embedding}.
	 Let $\iota_Q\colon Q\to D^{\E^c}$ and $\iota_M\colon M_D(\Sigma)\to M_D(\Sigma^c)$ be the embeddings. 
   By the Dirac Cross Section Theorem \ref{th:cross}, and using \eqref{eq:specialcase},   we obtain a Dirac morphism 
   \begin{equation}
    \iota^!R^c\colon (\T M_D(\Sigma),\,T M_D(\Sigma))\da \iota_Q^!(\AA^{\E^c},\ D^{\E^c}\times \dd^{\V^c})
   \end{equation}
   with base map $\Phi^c\circ \iota_M\colon M_D(\Sigma)\to Q$. The submanifold $Q$ is a direct product 
  \[ Q=D^\E\times \prod_{(\ez',\ez)}  Q_{\ez',\ez},\]
  where $\prod_{(\ez',\ez)}$ is a product over pairs of \emph{glued} edges, and $\iota_{\ez',\ez}\colon Q_{\ez',\ez}\hra D\times D$
  is the anti-diagonal (consisting of pairs $(d',d)$ with 
  $d'd=e$). Similarly,  
  \[ \iota_Q^!\AA^{\E^c}=\AA^\E\times \prod_{(\ez',\ez)}  \iota_{\ez',\ez}^!(\AA\times \AA).\] 
  For every pair $(\ez',\ez)$ of glued edges, consider the Dirac structure $\ti{S}_{\ez',\ez}\subset \AA\times \AA$, consisting of elements of the form $((d',\zeta_1,\zeta_2),(d,\zeta_2,\zeta_1))$ with $\zeta_i\in \dd$. Its image under the anchor is tangent to 
  $Q_{\ez',\ez}$, hence $\ti{S}_{\ez',\ez}|_{Q_{\ez',\ez}}\subset \a^{-1}(TQ_{\ez',\ez})$. 
 By reduction we obtain a Dirac structure 
  \[ S_{\ez',\ez}\subset  \iota_{\ez',\ez}^!(\AA\times \AA),\] 
  defining a Courant morphism $\iota_{\ez',\ez}^!(\AA\times \AA)\da 0$ to the zero Courant algebroid over a point.   Taking a product over all pairs of glued edges, together with the 
  identity morphisms for non-glued edges, we obtain a Courant morphism 
  $S\colon  \iota^!\AA^{\E^c}\da \AA^\E$
  with base map $\pi$. We claim that it defines a Dirac morphism
  \begin{equation}\label{eq:SDirac}
  S\colon \iota^!(\AA^{\E^c},\ D^{\E^c}\times \dd^{\V^c})\da (\AA^\E,\ D^\E\times \dd^\V).
  \end{equation}
  To see this, let $d^c\in  Q\subset D^{\E^c}$ with image $d=\pi(d^c)\in D^\E$. 
 Given $\zeta\in \dd^\V$, let  $\zeta^c\in\dd^{\V^c}$ be the quotient map $\V^c\to \V$ followed by 
  $\zeta\colon \V\to \dd$. Then $(d^c,\zeta^c)\sim_S (d,\zeta)$. Conversely, suppose $\zeta^c\in \dd^{\V^c}$ satisfies $(d^c,\zeta^c)\sim_S (d,\zeta)$ for \emph{some} $\zeta\in \dd^\V$.  Then the definition of $S$ shows that $\zeta^c$ is the quotient map 
  $\V^c\to \V$ followed by $\zeta$.   In particular, if $(d^c,\zeta^c)\sim_S (d,0)$ then $\zeta^c=0$.  Hence, \eqref{eq:SDirac} is a Dirac morphism as claimed.  By composition, we obtain the desired Dirac morphism 
 $R=S\circ \iota^! R^c$ with base map $\pi\circ \Phi^c\circ \iota_M=\Phi$. 
 
 For the last part of the theorem,  use the standard splittings to identify $\AA\cong \T_\eta D$. Note that the pullback of 
 $\pr_1^*\eta+\pr_2^*\eta$ to the anti-diagonal $Q_{\ez',\ez}\subset D\times D$ vanishes. Hence, 
 $\iota_{\ez',\ez}^!(\AA\times \AA)=\T Q_{\ez',\ez}$ is just a standard (untwisted) Courant algebroid. Furthermore, 
 in terms of this identification, $S_{\ez',\ez}=T Q_{\ez',\ez}$, and so the Courant morphism $\iota_{\ez',\ez}^!(\AA\times \AA)\da 0$ is just the standard Courant morphism $\T Q_{\ez',\ez}\da 0$ defined by the map $Q_{\ez',\ez}\to \pt$. It follows that 
 $S=\T\pi$. On the other hand, $R^c=\T_{\omega^c}\Phi^c$ where $\omega^c\in \Omega^2(M(\Sigma^c))$ is the 2-form for the 
 cut surface, and so $\iota^! R^c=\T_{\omega}(\Phi^c\circ \iota)$ with $\omega=\iota^*\omega^c$. Finally, 
 \[ R=\T\pi\circ \T_{\omega}(\Phi^c\circ \iota)=\T_\omega\Phi.\qedhere\]
\end{proof}

We may now consider the case of colored surfaces, as in Section \ref{sec:moduli2}. 
 
\begin{theorem}
	The space $(M_D^{\on{col}}(\Sigma),\omega^{\on{col}})$ is a quasi-symplectic $(G\times G)^{\E_{\on{free}}}$-space with $D^{\E_{\on{free}}}$-valued moment map \eqref{eq:coloredmoment}. Equivalently, there is an exact $(G\times G)^{\E_{\on{free}}}$-equivariant morphism of Manin pairs 
	\[  R^{\on{col}}\colon \left(\T M^{\on{col}}_D(\Sigma),TM^{\on{col}}_D(\Sigma)\right)\da \left(\AA^{\E_{\on{free}}} ,D^{\E_{\on{free}}}\times (\g\oplus \g)^{\E_{\on{free}}}\right).\]
In terms of the standard splitting of $\AA$, we have $ 	R^{\on{col}}=\T_{\omega^{\on{col}}}\Phi^{\on{col}}$. 
\end{theorem}

\begin{proof} 
The argument is similar to the proof of Theorem \ref{th:basic2}, hence we will be brief. Let $R$ be the Dirac morphism \eqref{eq:basic2} for the uncolored surface, with base map $\Phi\colon M(\Sigma)\to D^\E$. 
Let $Q\hra D^\E$ be the submanifold, consisting of maps $d\colon \E\to D$ taking colored edges into $G$. 
Then $M^c(\Sigma)=\Phi^{-1}(Q)$. Let $\iota_M\colon M^c(\Sigma)\to M(\Sigma)$ and 
$\iota_Q\colon Q\to D^\E$ be the embeddings. The Dirac Cross Section Theorem \ref{th:cross} gives a Dirac morphism
\[ \iota_Q^!R\colon 	\left(\T M^{\on{col}}_D(\Sigma),TM^{\on{col}}_D(\Sigma)\right)\da \iota_Q^!  (\AA^\E ,D^\E\times \dd^\V),\]
with base map $\Phi\circ \iota_M$. For each colored edge, let $\iota_\ez\colon G\to D$ be the inclusion. Then 
\[ \iota_Q^!\AA^{\E^c}=\AA^{\E_{\on{free}}}\times 
	\prod_{\ez\in  \E_{\on{col}}   }  \iota_{\ez}^!\AA.\] 
The image of the Dirac structure $\ti{S}_\ez=D\times \g\oplus \g\subset \AA$ under the 
anchor is tangent to $G$, hence it descends, by reduction,  to a Dirac structure $S_\ez\subset \iota_{\ez}^!\AA$, which we may regard as defining a Courant morphism $\iota_{\ez}^!\AA\da 0$. 
Taking the product over all colored edges, together with the identity morphism for the free edges, we obtain a Courant morphism 
$S\colon  \iota_Q^! \AA^\E \da \AA^{\E_{\on{free}}}$. In terms of the standard splitting, this is just $S=\T\pi$ for the projection 
$Q\to D^{\on{free}}$. By an argument parallel to that in the proof of Theorem \ref{th:basic2}, this 
is a Dirac morphism 
\[ S\colon \iota_Q^!  (\AA^\E ,D^\E\times \dd^\V)\da \left(\AA^{\E_{\on{free}}} ,D^{\E_{\on{free}}}\times (\g\oplus \g)^{\E_{\on{free}}}\right).\]
	We obtain $R^{\on{col}}$ as a composition of the Dirac morphisms, $R^{\on{col}}=S\circ \iota^!R$. 
\end{proof}

\end{appendix}

\bibliographystyle{amsplain} 
\def\cprime{$'$} \def\polhk#1{\setbox0=\hbox{#1}{\ooalign{\hidewidth
			\lower1.5ex\hbox{`}\hidewidth\crcr\unhbox0}}} \def\cprime{$'$}
\def\cprime{$'$} \def\cprime{$'$} \def\cprime{$'$} \def\cprime{$'$}
\def\polhk#1{\setbox0=\hbox{#1}{\ooalign{\hidewidth
			\lower1.5ex\hbox{`}\hidewidth\crcr\unhbox0}}} \def\cprime{$'$}
\def\cprime{$'$} \def\cprime{$'$} \def\cprime{$'$} \def\cprime{$'$}
\providecommand{\bysame}{\leavevmode\hbox to3em{\hrulefill}\thinspace}
\providecommand{\MR}{\relax\ifhmode\unskip\space\fi MR }
\providecommand{\MRhref}[2]{%
	\href{http://www.ams.org/mathscinet-getitem?mr=#1}{#2}
}
\providecommand{\href}[2]{#2}

\
\end{document}